\documentclass[10pt, a4paper]{amsart}
\usepackage{amsmath}
\usepackage{amssymb}
\usepackage{amsthm}
\usepackage{graphicx}
\usepackage{url}
\numberwithin{equation}{section}
\usepackage{nicefrac}
\usepackage{enumerate}
\usepackage{xcolor}
\usepackage{mathrsfs} 
\usepackage{mathtools}
\usepackage{t1enc}
\usepackage{lineno}
\newtheorem{thm}{Theorem}
\numberwithin{thm}{section}

\newtheorem{lem}[thm]{Lemma}

\newtheorem{theo}[thm]{Theorem}
\newtheorem{coro}[thm]{Corollary}
\newtheorem{propo}[thm]{Proposition}

\newtheorem{cla}[thm]{Claim}

\newtheorem*{theorem*}{Theorem}

\theoremstyle{definition}
\newtheorem{defn}[thm]{Definition}
\newtheorem{rem}[thm]{Remark}

\DeclareMathOperator{\supp}{supp} 

\DeclareMathOperator{\dist}{dist}

\DeclareMathOperator{\cardd}{card}
\newcommand{\card}[1]{\cardd\left(#1\right)}
\newcommand{\symdiff}{\div} 

\newcommand{\und}[1]{\underline{#1}}

\renewcommand{\phi}{\varphi}

\newcommand{\dbar}{\bar d}
\newcommand{\dund}{\underline{d}}

\DeclareMathOperator{\M}{\mathcal{M}}

\newcommand{\RR}{\mathbf{R}}

\newcommand{\eps}{\varepsilon}

\newcommand{\Z}{\mathbb{Z}}
\newcommand{\N}{\mathbb{N}}

\newcommand{\fD}{\mathfrak{D}}
\newcommand{\fF}{\mathfrak{F}}

\newcommand{\fd}{\mathfrak{d}}

\newcommand{\mru}{{\mathrm{u}}}

\newcommand{\mruu}{{\mathrm{uu}}}
\newcommand{\mrss}{{\mathrm{ss}}}
\newcommand{\mrc}{{\mathrm{c}}}

\def\MM{{\mathbb M}} \def\NN{{\mathbb N}}  
 \def\RR{{\mathbb R}}  
   
 \def\ZZ{{\mathbb Z}}

\def\La{\Lambda}

\def\Ga{\Gamma}

\newcommand{\eqdef}{\stackrel{\scriptscriptstyle\text{\rm def}}{=}}

\newcommand{\cA}{\mathcal{A}}\newcommand{\cD}{\mathcal{D}}
\newcommand{\cF}{\mathcal{F}}
\newcommand{\cI}{\mathcal{I}}
\newcommand{\cM}{\mathcal{M}}\newcommand{\cO}{\mathcal{O}}\newcommand{\cP}{\mathcal{P}}
\newcommand{\cR}{\mathcal{R}}\newcommand{\cS}{\mathcal{S}}\newcommand{\cT}{\mathcal{T}}
\newcommand{\cU}{\mathcal{U}}\newcommand{\cV}{\mathcal{V}}
\newcommand{\cZ}{\mathcal{Z}}

\def\diff{\operatorname{Diff}}
\def\Diff{\operatorname{Diff}}

\def\dim{\operatorname{dim}}

\def\supp{\operatorname{supp}}

\begin{document}

\title[Nonhyperbolic ergodic measures: positive entropy and full support]{Robust existence of nonhyperbolic ergodic measures with positive entropy and full support}
\date{\today}

\author[Ch.~Bonatti]{Christian Bonatti}
\address{Institut de Math\'ematiques de Bourgogne}
\email{bonatti@u-bourgogne.fr}

\author[L.~J.~D\'\i az]{Lorenzo J.~D\'\i az}
\address{Departamento de Matem\'atica, Pontif\'{\i}cia Universidade Cat\'olica do Rio de Janeiro} 
\email{lodiaz@mat.puc-rio.br}

\author[D.~Kwietniak]{Dominik Kwietniak}
\address{Faculty of Mathematics and Computer Science, Jagiellonian University in Krak\'ow}
\urladdr{http://www.im.uj.edu.pl/DominikKwietniak/}
\email{dominik.kwietniak@uj.edu.pl}

\begin{abstract}
We prove that for some manifolds $M$ the set of robustly transitive partially hyperbolic diffeomorphisms of $M$ with one-dimensional nonhyperbolic centre direction contains a $C^1$-open and dense subset of diffeomorphisms with nonhyperbolic measures which are ergodic, fully supported and have positive entropy. To do so, we formulate abstract conditions sufficient for the construction of an ergodic, fully supported measure $\mu$ which has positive entropy and is such that for a continuous function $\phi\colon X\to\mathbb{R}$ the integral $\int\phi\,d\mu$ vanishes. The criterion is an extended version of the \emph{control at any scale with a long and sparse tail} technique coming from the previous works. 
\end{abstract}

\begin{thanks}{This research has been supported  [in part] by CAPES - Ci\^encia sem fronteiras,
CNE-Faperj, and CNPq-grants (Brazil), the research of DK was supported by the National Science Centre (NCN) grant
2013/08/A/ST1/00275 and his stay in Rio de Janeiro, where he joined this project was possible thanks to
the CAPES/Brazil grant no. 88881.064927/2014-01.
The authors  acknowledge the hospitality of PUC-Rio and IM-UFRJ.
LJD thanks the hospitality and support of
ICERM - Brown University during the thematic semester 	
``Fractal Geometry, Hyperbolic Dynamics, and Thermodynamical Formalism''.}
\end{thanks}
\keywords{Birkhoff average,
entropy,
ergodic measure,
Lyapunov exponent,
nonhyperbolic measure,
partial hyperbolicity,
transitivity}
\subjclass[2000]{%
37D25, 
37D35, 
37D30, 
28D99
}


\maketitle



\section{Introduction}

Our motivation is the following problem:
\emph{To what extent does ergodic theory detect the nonhyperbolicity of a dynamical system?  Do nonhyperbolic dynamical systems always have a nonhyperbolic ergodic measure?}
These questions originated with a construction presented in \cite{GIKN} and inspired many papers exploring the properties of nonhyperbolic ergodic measures.

We emphasise that the answer to the second question be ``no'' in general: there are examples of nonhyperbolic diffeomorphisms
whose all ergodic measures are hyperbolic (even with all Lyapunov exponents uniformly bounded away from $0$), see \cite{BBS}. However,  these examples are ``fragile'': ergodic nonhyperbolic measures reappear after an arbitrarily small perturbation of these diffeomorphisms.

Thanks to the works of Abraham and Smale \cite{AbSm} and Newhouse \cite{New} it is known since late sixties that there exist open sets of nonhyperbolic systems. On the other hand, the first examples of open sets of diffeomorphisms with nonhyperbolic ergodic measures appeared just recently, in 2005 (see \cite{KN}). The construction in  \cite{KN} uses the \emph{method of periodic approximations} introduced in \cite{GIKN} (we will outline this method later). Note that this technique works in the specific setting of partially hyperbolic diffeomorphisms of the three-dimensional torus $\mathbb{T}^3$ with compact centre leaves.

The existence of nonhyperbolic ergodic measures for some nonhyperbolic systems raises immediately further questions, which we address here: \emph{Which nonhyperbolic dynamical systems have ergodic nonhyperbolic measures? What is the support of these measures?
What is their entropy? How many zero Lyapunov exponents do they have? How about other ergodic-theoretic properties of these measures?}

The following theorem is a simplified version of our main result. For the full statement see Theorems \ref{t.openanddense} and \ref{t.average} below.
\begin{theorem*}\label{thm:one}
For every $n\ge 3$ there are a  closed manifold $M$ of dimension
 $n$, a nonempty open set $\mathcal{U}$ in the space $\Diff^1(M)$ of $C^1$-diffeomorphisms defined
on $M$, and a constant $C>0$
such that every $f\in \mathcal{U}$ has a nonhyperbolic invariant measure $\mu$ (i.e., with some zero Lyapunov exponent) satisfying:
 \begin{enumerate}
  \item \label{tp1} $\mu$ is ergodic,
  \item \label{tp2} the support of $\mu$ is the whole manifold $M$,
  \item \label{tp3} the entropy of $\mu$ is larger than $C$.
 \end{enumerate}
\end{theorem*}
The Theorem applies to any manifold allowing a robustly transitive diffeomorphism which has partially hyperbolic splitting with one-dimensional centre. In particular, by \cite{BD-robtran} it applies to
\begin{itemize}
\item the $n$-dimensional torus ($n\ge 3$),
\item (more generally) any manifold carrying a transitive Anosov flow.
\end{itemize}

The nonhyperbolic measure $\mu$ as in the Theorem above apart from being ergodic, fully supported and having positive entropy, is also \emph{robust}. Namely, our proof shows that the three former properties appear \emph{robustly} in the space $\Diff^1(M)$ (i.e. we provide an open set of diffeomorphisms with an ergodic nonhyperbolic fully supported measure of positive entropy). Several previous works established the existence of
nonhyperbolic measures in similar settings as in Theorem above and these measures have some (but not all) of the properties listed above. These references present two approaches to the construction of ergodic nonhyperbolic measures: the first one is the already mentioned \emph{method of periodic approximations} \cite{GIKN}, the second is the method of generating a nonhyperbolic measure by a \emph{controlled point} \cite{BBD:16,BDB:}. Note that  \cite{BZ} combines these two methods. Both schemes are detailed in Section \ref{ss.compar}. A discussion of  previous results on nonhyperbolic ergodic measures see Section \ref{ss.history} or \cite{D-ICM} for a more comprehensive survey on that topic.

Here we further extend the method of construction of a controlled point presented in \cite{BDB:} and complement it by an abstract result in ergodic theory (see the theorem about entropy control and the discussion in Section \ref{ss.control}). These allows us to address simultaneously all four properties listed above.

\subsection{Precise results for robustly transitive diffeomorphisms}\label{ss.precise}
In what follows $M$ denotes a closed compact manifold, $\Diff^1(M)$ is the space of $C^1$-diffeo\-morphisms
of $M$ endowed with the usual $C^1$-topology, and $f\in \Diff^1(M)$. For $\Lambda\subset M$ we write
$\cM_f(\Lambda)$ for the set of all $f$-invariant measures with support contained in $\Lambda$.
A  $Df$-invariant splitting $T_\La M=E\oplus F$ is \emph{dominated} if
there are constants $C>0$ and $\lambda<1$ such that
$\| Df^{n} E_{x}\|\cdot \| Df^{-n} F_{f^{n}(x)}\| < C  \lambda^n$
for every $x\in\Lambda$ and $n\in \NN$.
We say that a compact $f$-invariant set $\Lambda$
is \emph{partially hyperbolic with one-di\-men\-sio\-nal center}
if there is a $Df$-invariant
splitting with three nontrivial bundles\begin{equation}\label{e.ph}
T_\Lambda M = E^{\mrss} \oplus E^{\mrc} \oplus E^{\mruu}
\end{equation}
 such that
$E^\mathrm{ss}$ is uniformly contracting and
	$E^\mathrm{uu}$ is uniformly expanding, $\dim E^{\mrc}=1$, and the $Df$-invariant
splittings
$E^{\mathrm{cs} } \oplus E^{\mruu}$
and
$E^{\mrss} \oplus E^{\mathrm{cu}}$ are both dominated,
where
$E^\mathrm{cs}= E^{\mrss} \oplus E^{\mrc}$ and  $E^\mathrm{cu}= E^{\mrc} \oplus E^{\mruu}$.
The bundles
	$E^{\mruu}$ and $E^{\mrss}$ are called \emph{strong stable} and \emph{strong unstable},
	respectively,  and
	$E^\mrc$ is the \emph{center bundle}. We abuse a bit the terminology and say that the splitting given by \eqref{e.ph} is also \emph{dominated}.

 If $\Lambda$ is a partially hyperbolic set with one-dimensional center, then
the bundles $E^{\mrss}, E^{\mrc}, E^{\mruu}$ depend continuously on the point $x\in \Lambda$. Hence the
\emph{logarithm of the center derivative}
\begin{equation}\label{e.logmap}
 \mathrm{J}_f^{\mrc}(x) \eqdef \log | Df_x |_{E^\mrc (x)\setminus\{0\}}|
\end{equation}
is a continuous map. If, in addition, $\mu\in \cM_f(\Lambda)$ is ergodic, then  the  Oseledets Theorem
implies that there is a
number  $\chi^{\mrc} (\mu)$, called the {\emph{central Lyapunov exponent of $\mu$,}}
such that for $\mu$-almost every point $x\in \Lambda$ it holds
$$
\lim_{n \to \pm \infty} \frac{\log |Df^n_x (v)|}{n} = \int  \mathrm{J}_f^\mrc \, d\mu=\chi^{\mrc}(\mu), \qquad
\mbox{for every $v\in E^{\mrc} \setminus \{0\}$}.
$$
In particular, the function $\mu\mapsto\chi^{\mrc}(\mu)$ is continuous with respect to the weak$^*$ topology on $\cM_f(\Lambda)$.

A diffeomorphism $f\in \diff^1(M)$ is
\emph{transitive} if it has a dense orbit. The diffeomorphism $f$ is
\emph{$C^1$-robustly transitive}
if it belongs to the $C^1$-interior of the set of transitive diffeomorphisms
(i.e., all $C^1$-nearby diffeomorphisms are also transitive).

We denote by $\cR\cT(M)$ the $C^1$-open set of such diffeomorphisms $f\in\diff^1(M)$ that:
\begin{itemize}
\item
$f$ is robustly transitive,
\item
$f$ has a pair  of hyperbolic periodic points with different indices,
\item $M$ is a partially hyperbolic set for $f$ with one-dimensional center.
\end{itemize}
These assumptions imply that
$\dim(M)\ge 3$, because in dimension two robustly transitive diffeomorphisms are
always hyperbolic, see \cite{PuSa}. 

The set $\cR\cT(M)$ contains
well studied and interesting examples of diffeomorphisms. Among them there are:
different types of skew product diffeomorphisms, see
 \cite{BD-robtran,Sh};
 derived from Anosov diffeomorphisms, see \cite{Mda};
and perturbations of  time-one maps of transitive Anosov flows, see
\cite{BD-robtran}.

Our main theorem provides a measure $\mu$ which is ergodic \emph{and} has full support \emph{and} positive entropy, \emph{and} is such that
the integral $\int J_f^{\mrc}d\mu$ has a prescribed value ($=0$).

\begin{thm}
\label{t.openanddense}
There is a $C^1$-open and dense subset $\cZ(M)$ of $\cR\cT(M)$ such that
every $f\in \cZ(M)$ has an ergodic nonhyperbolic measure
with positive entropy and full support. Furthermore, for each
$f\in \cZ(M)$ there is a neighbourhood $\cV_f$ of $f$ and a constant $c_f>0$  such that
for every $g\in\cV_f\cap\cZ(M)$ we have $h(\mu_g)\ge c_f$.
\end{thm}


In what follows, given a periodic point $p$ of a diffeomorphism $f$ we denote by
$\mu_{\cO(p)}$ the unique $f$-invariant measure supported on the orbit ${\cO(p)}$ of $p$.
The next result is a reformulation and  an extension of Corollary 6 in \cite{BDB:} to our context:

\begin{thm}\label{t.average}
Consider an open subset $\cU$ of $\cR\cT(M)$
such that there is a continuous map defined on $\cU$
\[
f\mapsto (p_f,q_f)
\]
that associates to each
$f\in \cU$ a pair of hyperbolic periodic points with $\cO(p_f)\cap\cO(q_f)=\emptyset$.
 Let $\varphi\colon M\to\RR$ be a continuous function
 satisfying
$$
\int \varphi\, d\mu_{\cO(p_f)}<0<\int \varphi \,d\mu_{\cO(q_f)}, \quad \mbox{for every $f\in \cU$}.
$$

Then there is a  $C^1$-open
and dense
set $\cV$  of $ \cU$ such that every $f\in\cV$  has an ergodic measure
$\mu_f$ whose support is $M$, satisfies
$\int \varphi\, d\mu_f=0$, and has positive entropy. Furthermore, for each
$f\in \cV$ there is a neighbourhood $\cV_f$ of $f$ and a constant $c_f>0$  such that
for every $g\in\cV_f\cap\cV$ we have $h(\mu_g)\ge c_f$.
\end{thm}
\begin{rem}
Note that in Theorem~\ref{t.average} there is no condition on the indices of the periodic points.
Observe  also that Theorem~\ref{t.openanddense} is not a particular case of
Theorem~\ref{t.average}:
in Theorem~\ref{t.openanddense} there is no fixed map $\varphi$,
the considered map  is the logarithm of the center derivative $J_f^{\mrc}$ and therefore it changes with $f$ (i.e.,
there is a family of maps $\varphi_f$ depending on $f$).
\end{rem}

\begin{rem}
We could state a version of Theorem \ref{t.openanddense} replacing the logarithm of the center derivative $J_f^{\mrc}$ by any
continuous map $\varphi\colon M\to \mathbb{R}$ and picking any value $t$ satisfying
\[
\inf\left\{ \int \varphi\, d\mu\,: \, \mu \in \cM_f(M)\right\}<\, t\, <
\sup\left\{ \int \varphi\, d\mu\,:\, \mu \in \cM_f(M)\right\}.
\]
\end{rem}

\begin{rem}
Our methods imply that
 the sentence ``{\emph{Moreover, the measure $\mu$ can be taken with positive entropy.}}''
 can be added to Theorems 5, 7, and 8   and Proposition 4b in \cite{BDB:}.
Furthermore, the entropy will have locally a uniform lower bound.
 The details are left to the reader.
\end{rem}
\subsection{Previous results on nonhyperbolic ergodic measures}\label{ss.history}
By \cite{CCGWY} nonhyperbolic homoclinic classes of $C^1$-generic diffeomorphisms always support nonhyperbolic ergodic measures. The proof uses the periodic approximation method and extends \cite{DG}. In some settings the results of \cite{BDG} imply that these measures have full support in the homoclinic class.

Specific examples of open sets of diffeomorphisms of the three torus with nonhyperbolic ergodic measures were first obtained in \cite{KN} using the periodic approximation method. In \cite{BBD:16} there are general results  guaranteeing for an open and dense subset of $C^1$-robustly transitive diffeomorphisms the existence of nonhyperbolic ergodic measures with positive entropy. The latter paper uses the controlled point method.

All results above provide measures with only one zero Lyapunov exponent. In \cite{WZ} yet another adaptation of the method of periodic approximation yields ergodic measures with multiple zero Lyapunov exponents for some $C^1$-generic diffeomorphisms (see also \cite{BBD:14} for results about skew-products).


Some limitations of these previous constructions are already known. By \cite{KL} any measure obtained by the method of periodic approximations has zero entropy. Hence all nonhyperbolic measures defined in \cite{BDG,BZ,CCGWY,DG,GIKN, KN,WZ} have necessarily zero entropy. On the other hand, the measures produced in \cite{BBD:16} cannot have full support for their definition immediately implies that they are supported on a Cantor-like subset of the ambient manifold.

\subsection{Comparison of  constructions of nonhyperbolic ergodic measures}
\label{ss.compar}
So far there are two ways to construct nonhyperbolic ergodic measures: these measures are either ``approximated by periodic measures'' as in \cite{GIKN} or ``generated by a controlled point'' as in \cite{BBD:16,BDB:}.
 Both methods apply to any partially hyperbolic diffeomorphism as in Subsection~\ref{ss.precise} whose domain contains two special subsets:
a centre contracting region
(where the central direction is contracted)
and a
centre expanding  region (where the central direction is expanded). Furthermore, the orbits can travel from one of these regions to the other in a controlled way. For example, both methods work for a diffeomorphism with a transitive set (which is persistent to perturbations
if one wants to obtain a robust result)  containing two ``heteroclinically related blenders'' of different
indices and a region where the dynamics is partially hyperbolic with one dimensional centre. We discuss blenders in Section \ref{s.robustly}.

Let us briefly describe these two approaches. For simplicity, we will restrict ourselves to the partially hyperbolic setting as above. 

As mentioned above the ``approximation by periodic orbits'' construction from \cite{GIKN}
 builds a sequence of periodic orbits
$\Gamma_i$ and periodic measures $\mu_i$ supported on these orbits such that the sequence of central Lyapunov exponents $\lambda_i$ of these measures tends to zero as $i$ approaches infinity. Since in our partial hyperbolic setting the central Lyapunov exponent vary continuously with the measure, it vanishes for
every  weak$^*$ accumulation point $\mu$ of the sequence of measures $\mu_i$ so we get a nonhyperbolic measure.
The difficulty is to prove that every limit measure is ergodic as ergodicity is not closed property in the weak$^*$ topology.
The arguments in \cite{GIKN} contain a general criterion for the weak$^*$ convergence and ergodicity of the limit of a sequence periodic orbits $\Gamma_i$. It requires that for some summable sequence $(\gamma_i)$ of positive reals most points on the orbit $\Gamma_{i+1}$ shadow $\gamma_i$-close a point on the precedent orbit $\Gamma_i$ for $\card{\Gamma_i}$ iterates. If the proportion of these shadowing points among all points of $\Gamma_{i+1}$ tends to $1$
as $i\to +\infty$ sufficiently fast, then $\mu_i$ weak$^*$ converge to an ergodic measure. The ``non-shadowing'' points are used to decrease the absolute value of the central Lyapunov exponent over $\Gamma_i$ and to spread the support of $\Gamma_i$ in the ambient space.
This forces the limit measure to have zero central Lyapunov exponent and full support, as in this case the central Lyapunov exponent depends continuously on the measure (see below).


It turns out that the repetitive nature of the method of periodic approximation forces the resulting measure to be close (Kakutani equivalent\footnote{Two measure preserving systems are \emph{Kakutani equivalent} if they have a common derivative. A \emph{derivative} of a measure preserving system is an another measure preserving system isomorphic with a system induced by the first one. See Nadkarni's book \cite{Nadkarni}, Chapter 7.})  to a group rotation. Actually, it is proved in \cite{KL} that the periodic measures $\mu_i$ described above converge to $\mu$ in a much stronger sense than weak$^*$ convergence. This new notion of convergence is coined \emph{Feldman--Katok convergence} and it implies that all measures obtained following \cite{GIKN} (thus the measures from \cite{ BDG,  BZ,CCGWY, DG, KN, WZ} obtained by this method) are loosely Kronecker measures with zero entropy. For more details we refer to \cite{KL}.



The authors of \cite{BBD:16} devised a new method for constructing nonhyperbolic ergodic measures
using blenders and flip-flop configurations (we review the former in Section \ref{s.robustly} and the latter in Section~\ref{s.flipflop}). Applying these tools one defines a point $x$ such that the Birkhoff averages of the central derivative along
segments of its forward orbit go to zero uniformly. In a bit more precise terms,
there are sequences of positive reals and positive integers, denoted $\varepsilon_n$ and $T_n$, with $\varepsilon_n\to 0^+$
and $T_n\to \infty$ as $n\to \infty$ such that the average of the central derivative
along a segment of the $x$-orbit
$\{f^t(x), \dots,f^{t+T_n}(x)\}$ is less than $\varepsilon_n$ for any $t\geq 0$. We say that such an $x$ is \emph{controlled at any scale}.
Then the $\omega$-limit set $\omega(x)$ of $x$ is an invariant compact set such that for all measures supported on $\omega(x)$ the centre Lyapunov exponent vanishes,
see  \cite[Lemma 2.2]{BBD:16}. 
Under some mild assumptions one can find a point $x$ such that the compact invariant set $\omega(x)$ is also partially hyperbolic and $f|_{\omega(x)}$ has the full shift over a finite alphabet as a factor, thus it has positive topological entropy. To achieve this one finds a pair of disjoint compact subsets $K_0$ and $K_1$ of $M$ such that for some $k>0$ and $\omega\in\{0,1\}^\ZZ$ which is generic for the Bernoulli measure $\xi_{1/2}$ for every $j\in\ZZ$ we have $f^{jk}(x)\in K_{\omega(j)}$. By the variational principle for topological entropy \cite{Wal:82} the set $\omega(x)$
supports an ergodic nonhyperbolic measure with positive entropy.
Unfortunately, the existence of a semi-conjugacy from $\omega(x)$ to a Cantor set carrying the full shift forces $\omega(x)$ to be a proper subset of $M$, thus these measures cannot be supported on the whole manifold.



In  \cite{BDB:} the procedure from \cite{BBD:16} was modified.
More precisely, in \cite{BDB:} the control over orbit of a point $x$ is relaxed:
one splits the orbit of $x$ into a ``regular part'' and a ``tail'', and
one only needs  to control the averages over the orbit segment $\{f^t(x), \dots,f^{t+T_n}(x)\}$ of length $T_n$ only for $t$ belonging to the regular part, which is a set of positive density in $\N$ and at the same time the iterates corresponding to the tail part are dense in $X$.
Under quite restrictive conditions on the tail (coined \emph{longness} and  \emph{sparseness}),  Theorem~1 from \cite{BDB:} claims that if $x$ is
controlled at any scale with a long sparse tail then any  measure $\nu$ generated\footnote{A measure $\mu$ is generated by a point $x$ if $\mu$ is a weak$^*$ limit point of the sequence of measures $\frac1n\sum_{i=0}^{n-1}\delta_{f^i(x)}$,
 where $\delta_{f^i(x)}$ is the Dirac measure at the point $f^i(x)\in X$.} by $x$ has vanishing central Lyapunov exponent and full support.
The underlying topological mechanism  providing  points whose orbits are controlled at any scale with long sparse tail are
the \emph{flip-flop families with sojourns in $X$}.

Since there is no longer a semi-conjugacy to a full shift a different method has to be applied to establish positivity of the entropy.

\subsection{Control of entropy}\label{ss.control}
In this paper we combine the two methods above. We pick a pair of disjoint compact subsets $K_0$ and $K_1$ of $M$ and we divide the orbit into the regular and tail part as in \cite{BDB:}. We assume that the controlled points visits $K_0$ and $K_1$ following the same pattern as some point generic for the Bernoulli measure $\xi$ (as in \cite{BBD:16}), but we require that this happens \emph{only} for iteration in the regular part of the orbit. We also assume that the tail is even more structured: apart of being long and sparse, the tail intersected with nonnegative integers is a \emph{rational subset of $\NN$}. A set $A\subset\NN$ is \emph{rational} if it can be approximated with arbitrary precision by sets which
are finite union of arithmetic progressions (sets of the form $a+b\NN$, where $a,b\in\NN$, with $b\neq 0$). Here, the ``precision'' is measured in terms of the upper asymptotic density $\bar{d}$ of the symmetric difference of $A$ and a finite union of arithmetic progressions.


To get positive entropy, we show that
the measure $\mu$  is an extension of a loosely Bernoulli system with positive entropy
(a measure preserving system with a subset of positive measure such that the induced system is a Bernoulli process).

More precisely, we have the following general criterion for the positivity of any measure generated by a point (actually, we prove even more general Theorem \ref{thm:mainbis}, but for the full statement we need more notation, see Section \ref{s.measurespositive}).

\begin{theorem*}[Control of entropy]\label{thm:main}
Let $(X, \rho) $ be a compact metric space and $f\colon X\to X$ be a continuous map.
Assume that $\mathbf{K}=(K_0,K_1)$ is a pair of disjoint compact subsets of $X$ and $J\subset \N$ is
a rational.
If $\bar x\in X$ is such that $f^j(\bar x)\in K_{\bar z(j)}$ for every $j\in J$ where $z$ is a generic point for the Bernoulli measure $\xi_{1/2}$, then the entropy of
any measure $\mu$ generated by $x$ has entropy at least $d(J)\cdot \log 2$,
where $d(J)$ stands for the asymptotic density of $J$.
\end{theorem*}

We apply the above criterion to the controlled point and the rational set $J$ which is the complement (in $\NN$) of the intersection of the rational tail with $\N$. This implies that $J$ is also a rational set (Remark~\ref{r.complement}) and the asymptotic density of $J$ exists and satisfies $0<d(J)<1$.

The underlying topological mechanism we use to find the controlled point with the required behavior
is provided by a new object we call
 the
\emph{double flip-flop family for $f^N$ with $f$-sojourns in $X$}, where $N>0$ is some integer.
This mechanism is a variation of the notion of the flip-flop family with sojourns in $X$, indeed
we will see  that a flip-flop family with sojourns in $X$  yields
a double flip-flop family for some $f^N$ with $f$-sojourns in $X$, see Proposition~\ref{p.yieldsdouble}.




 \subsection*{Organisation of the paper}

This paper is organised as follows. In Section~\ref{s.measurespositive}, we prove Theorem~\ref{thm:main}.
Section~\ref{s.scalesandtails} is devoted to the construction of rational sparsely long tails. In Section~\ref{s.flipflop}, we
study different types of flip-flop families and prove that they generate ergodic measures  with full support and positive entropy
with appropriate averages.
Finally, in Section~\ref{s.robustly}
devoted to robustly transitive diffeomorphisms, we complete the proofs of Theorems~\ref{t.openanddense} and \ref{t.average}.





\section{Measures with positive entropy. Proof of Theorem~\ref{thm:main}}
\label{s.measurespositive}
The goal of this section is to prove Theorem~\ref{thm:main}.
Throughout what follows, $X$ is a compact metric space, $\rho$ is a metric for $X$, and $f\colon X\to X$ is a continuous map (not necessarily a homeomorphism).  First we introduce some notation.

Given a continuous function $\varphi\colon X\to \RR$, $n>0$, and $x\in X$ we denote by $\varphi_n(x)$
the Birkhoff average of $\varphi$ over the orbit segment $x,f(x),\ldots,f^{n-1}(x)$, that is,
 \begin{equation}\label{e.averages}
 \phi_n(x)\eqdef  \frac 1n\sum_{i=0}^{n-1}\varphi\circ f^i(x). 
 \end{equation}

\subsection*{Generated measures, generic/ergodic points} Let $\M_f(X)$ denote the set of $f$-invariant Borel probability measures on $X$. Given $x\in X$ and $n\ge 1$ we set
\begin{equation}\label{e.empiric}
 \mu_n(x,f)\eqdef \frac1n\sum_{i=0}^{n-1}\delta_{f^i(x)},
 \end{equation}
 where $\delta_{f^i(x)}$ is the Dirac measure at the point $f^i(x)\in X$.
We say that a point $x\in X$ {\emph{generates $\mu\in\M_f(X)$} along $(n_k)_{k\in\N}$} if $(n_k)_{k\in\N}$ is a strictly increasing sequence of integers such that $\lim_{k\to\infty}\mu_{n_k}(x,f)=\mu$ in the weak$^*$ topology on the space of probability measures on $X$. If $\mu$ is a measure generated by a point $x\in X$, then $\mu$ is invariant, that is $\mu\in\cM_f(X)$. If there is no need to specify $(n_k)$ we just say that \emph{$x$ generates $\mu$}. We write
$V(x)$ for the set of all $f$-invariant measures generated by $x$.
We say that $x$ is a \emph{generic point} for $\mu$ if $\mu$ is the unique measure generated by $x$,
i.e., $V(x)=\{\mu\}$.

A point is an \emph{ergodic point} if it is a generic point for an ergodic measure.
We write $h_\mu(f)$ for
the \emph{entropy} of $f$ with respect to $\mu$, see  \cite{Wal:82} for its
definition and basic properties.

\subsection*{Sets and their densities}
The \emph{upper asymptotic density} of a set $J\subset \Z$ ($J\subset\N$) is the number
\[
\dbar(J)=\limsup_{N\to\infty}\frac{1}{N}\card{J\cap [0,N-1]}.
\]
Similarly, we define the \emph{lower asymptotic density} $\dund(J)$ of $J$. If $\dbar(J)=\dund(J)$, then we say that $J$ has \emph{asymptotic density} $d(J)=\dund(J)=\dbar (J)$. Note that $\dbar$, $\dund$, and $d$ (if defined) are determined only by $J\cap\N$. 

\subsection*{Symbolic dynamics} Let $\Omega_M=\{0,1,\ldots,M-1\}^\Z$ be the full shift over $\cA=\{0,1,\ldots,M-1\}$.
For $\alpha\in\cA$ we write $[\alpha]_0$ for the \emph{cylinder set} defined as $\{\omega\in\Omega_M:\omega_0=\alpha\}$. Similarly, given a word $u=\alpha_1\dots \alpha_k\in\cA^k$ the set
$[u]_0$ is the cylinder defined as $\{\omega\in\Omega_M:\omega_0=\alpha_1, \dots,
\omega_{k-1}=\alpha_k\}$. We will often identify a set $A\subset\Z$ ($A\subset\N$) with its characteristic function $\chi_A\in\Omega_2$,
that is $(\chi_A)_i=1$ if and only if $i\in A$. By $\sigma$ we denote the shift homeomorphism on $\Omega_M$ given by $(\sigma(\omega))_i=\omega_{i+1}$ for each $i\in \Z$.
For more details on symbolic dynamics we refer the reader to \cite{LM}.

\subsection*{Completely deterministic sequences}
A point $x\in\Omega_2$ is \emph{completely deterministic} (or \emph{deterministic} for short) if every measure generated by $x$ has zero entropy, that is, $h(\mu)=0$ for every $\mu\in V(x)$. A set $J\subset \Z$ is \emph{completely deterministic} if its characteristic function is a completely deterministic point in $\Omega_2$. This notion is due to B.~Weiss, see \cite{Weiss}.

\subsection*{Bernoulli measure} The {\emph{Bernoulli measure $\xi_{1/2}$}} is the shift invariant measure on $\Omega_2$ such that for each $N\in\mathbb{N}$ and $u=u_1\ldots u_N\in\{0,1\}^N$ we have $\xi_{1/2}([u]_0)=1/2^N$.


\subsection*{Itineraries}
Let $\mathbf{K}=(K_0,K_1)$
be disjoint compact subsets of $X$ and $J\subset\Z$.
We say that $\omega\in\Omega_2$ is the $\mathbf{K}$-itinerary of $x\in X$ over $J$ if $f^{j}(x)\in K_{\omega(j)}$ for each $j\in J$.

The next result is the main step in the proof of Theorem~\ref{thm:main}.

\begin{thm}[Control of the entropy]\label{thm:mainbis}
Let $(X, \rho) $ be a compact metric space and $f\colon X\to X$ be a continuous map.
Assume that $\mathbf{K}=(K_0,K_1)$ is a pair of disjoint compact subsets of $X$ and $J\subset \N$ is
completely deterministic with $\dund(J)>0$.
If the $\mathbf{K}$-itinerary of $\bar x\in X$ over $J$ is a generic point for $\xi_{1/2}$, then the entropy of
any measure $\mu\in V(x)$ satisfies
$$
h _{\mu}(f)\ge \dund(J)\cdot \log 2>0.
$$
\end{thm}

Our proof is based on the following property of completely deterministic sets:
\emph{A sequence formed by symbols chosen from a generic point of a Bernoulli measure $\xi_{1/2}$ along a completely deterministic set with positive density is again a generic point for $\xi_{1/2}$}. This is a result of Kamae and Weiss (see \cite{Weiss}) originally formulated in the language of normal numbers and admissible selection rules.

\begin{rem}
When we were finishing writing this paper, {\L}{\c{a}}cka announced in her PhD thesis \cite{Martha} a version of Theorem \ref{thm:mainbis} with relaxed assumptions on the point $\bar x$ and the sequence $J$. Her proof is based on properties of the $\bar{f}$-pseudometric discussed in \cite{KL}.
\end{rem}

\subsection{Proof of Theorem~\ref{thm:mainbis}} 
Let $K_2=X\setminus (K_0\cup K_1)$. Let $\iota_\mathcal{P}\colon X\to\Omega_3$ be the coding map with respect to the partition $\mathcal{P}=\{K_0,K_1,K_2\}$. In other words, $\iota_\mathcal{P}(x)=y\in\Omega_3$, where
\[
\iota_\mathcal{P}(x)=\begin{cases}
                       i, & \mbox{if $j\in\N$ and $f^j(x)\in K_i$}, \\
                       0, & \mbox{if $j<0$}.
                     \end{cases}
\]
Fix $\nu\in V(\bar x)$. We modify $K_0$ and $K_1$ without changing the $\mathbf{K}$-itinerary of $\bar x$ over $J$ so that the topological boundary of $K_i$ for $i=0,1,2$, denoted $\partial K_i$, is $\nu$-null. To this end, we set
\[
\tilde{c}=\frac{1}{2}\min\{\rho(x_0,x_1):x_0\in K_0,\,x_1\in K_1\}
\]
and for $0<c<\tilde{c}$ we define sets $\partial_c K_i =\{x\in X: \dist (x,K_i)=c\}$ for $i=0,1$. Note that for each $c$ the set $\partial_c K_0\cup\partial_c K_1$ contains (but need not to be equal) the topological boundaries of the sets: $K^c_0$, $K^c_1$, and $K'_2=X\setminus (K^c_0\cup K^c_1)$, where $K^c_i=\{x\in X: \dist (x,K_i)\le c\}$ for $i=0,1$. Consider a family of closed sets $\mathcal{C}=\{\partial_c K_0\cup\partial_c K_1: 0<c<\tilde c\}$.
Since elements of $\mathcal{C}$ are pairwise disjoint, only countably many of them can be of positive $\nu$-measure.
Fix any $0<c<\tilde{c}$ such that $\nu(\partial_c K_0\cup\partial_c K_1)=0$ and replace $K_i$ by $K^c_i$ for $i=0,1$, and $K_2$ by $K'_2$. Note that this does not change the $\mathbf{K}$-itinerary of $x$ over $J$.  Furthermore, the elements of our redefined partition, which we will still denote $\mathcal{P}=\{K_0,K_1,K_2\}$ have $\nu$-null boundaries. Let $(n_k)$ be the sequence of integers along which $\bar x$ generates $\nu$. Then $y\eqdef\iota_\mathcal{P}(\bar x)\in\Omega_3$ generates a shift-invariant measure $\mu$ on $\Omega_3$ which is the push-forward of $\nu$ through the coding map $\iota_\mathcal{P}$. Furthermore, the dynamical entropy of $\nu$ with respect to the partition $\mathcal{P}$ equals $h_\mu(\sigma)$. See \cite[Lemma 2]{KL} for more details. Note that the proof in \cite{KL} is stated for generic points, but it is easily adapted to the measures generated along a sequence as considered here. Now Theorem~\ref{thm:mainbis} follows from the following fact.

\begin{cla}
\label{t.p.entropy}
If $\mu \in \mathcal{M}_\sigma(\Omega_3)$ is a measure
generated by $y$, then
$h_{\mu}(\sigma) \ge \dund(J)\cdot \log 2$.
\end{cla}




\begin{proof}[Proof of the Claim \ref{t.p.entropy}]
Let $\mu\in V(y)$ and $(n_k)_{k\in \mathbb{N}}$ be a strictly increasing sequence such that $\mu_{n_k}(y,\sigma)\to\mu$ as $k\to\infty$.

Let $\chi_J\in\Omega_2$ be the characteristic sequence of $J$.  Let $\mu_J$ by any measure generated by $\chi_J$ along
$(n_k)_{k\in \mathbb{N}}$, that is, is any limit point of $(\mu_{n_k}(\chi_J,\sigma))_{k\in\N}$. Passing to a subsequence (if necessary) we assume that $\mu_{n_k}(\chi_J,\sigma)$ converges as $k\to\infty$ to a shift-invariant measure $\mu_J$ on $\Omega_2$.
Consider the product dynamical system on $\Omega_3\times \Omega_2$ given by
$$
S \eqdef
\sigma\times \sigma\colon \Omega_3\times \Omega_2\to \Omega_3\times \Omega_2.
$$
Again passing to a subsequence if necessary, we may assume that there exists $\mu'\in\mathcal{M}_{S}(\Omega_3\times\Omega_2)$ such that
\[
\mu_{n_k}((y,\chi_J),S)
\to \mu'\text{ as }k\to\infty.
\]
Recall that a \emph{joining} of $\mu$ and $\mu_J$ is an $S$-invariant measure on $\Omega_3\times \Omega_2$ which projects $\mu$ in the first coordinate and $\mu_J$ in the second. Observe that $\mu'$ is a joining of $\mu$ and $\mu_J$ (because the marginal distributions of $\mu_{n_k}((y,\chi_J),S)$ converge as $k\to\infty$ to, respectively, $\mu$ and $\mu_J$).
As the entropy of a joining is bounded below by the entropy of any of its marginals and is bounded above by the sum of the entropies of its marginals (see \cite[Fact 4.4.3]{Downarowicz}) we have that
\[
h_\mu(\sigma)\le h_{\mu'}(S)\le h_\mu(\sigma)+h_{\mu_J}(\sigma)=h_\mu(\sigma),
\]
where the equality uses that $J$ is completely deterministic. 
Now to complete the proof the Claim \ref{t.p.entropy} it suffices to show the
following fact.

\begin{cla}
\label{p.l.entropy}
 Every  $S$-invariant measure $\mu'\in V(y,\chi_J)$ satisfies $h_{\mu'}(S)\ge \und d(J)\cdot \log 2>0$.
\end{cla}

\begin{proof}[Proof of Claim \ref{p.l.entropy}]
Let $\Psi\colon \{0,1,2\}\times\{0,1\}\to\{0,1,2\}\times\{0,1\}$ be the $1$-block map given by
$\Psi(\alpha,1)=(\alpha,1)$ and $\Psi(\alpha,0)=(2,0)$ for $\alpha\in\{0,1,2\}$.
Consider the factor map $\psi\colon\Omega_3\times\Omega_2\to\Omega_3\times\Omega_2$ determined by $\Psi$, that is
\[
\psi(\omega)=\big((\psi(\omega))_i\big)_{i\in\Z},\qquad\text{where }(\psi(\omega))_i = \Psi(\omega_i) \,\text{for }i\in\Z.
\]
Observe that we have defined $\psi$ so that if $\omega,\omega'\in \Omega_3$ satisfy $\omega|_J=\omega'|_J$, then $(\omega'',\chi_J)\eqdef\psi(\omega,\chi_J)=\psi(\omega',\chi_J)$. Furthermore, $\omega''$ agrees with both, $\omega$ and $\omega'$, over $J$ and $\omega''_j=2$ for all $j\notin J$.
In particular, if $\bar z\in\Omega_2$ is a generic point for the Bernoulli measure $\xi_{1/2}$ such that
$\bar z|_J=
y|_J$,
then  $(z,\chi_J)\eqdef\psi(y,\chi_J)=\psi(\bar z,\chi_J)$.

Recall that the only joining of a Bernoulli measure $\xi_{1/2}$ and the zero entropy measure $\mu_J$ is the product measure $\xi_{1/2}\times\mu_J$, see \cite[Theorem 18.16]{Glasner}.
 As any limit point of $(\mu_{n_k}((\bar z,\chi_J),S))_{k\in\N}$ is a joining of $\xi_{1/2}$ and $\mu_J$, we get that
 $(\bar z,\chi_J)$ generates along $(n_k)_{k\in\N}$ the $S$-invariant measure $\xi_{1/2}\times \mu_J$. It follows that $(z,\chi_J)$ generates  along $(n_k)_{k\in\N}$ the $S$-invariant measure
 $\mu''=\psi_*(\xi_{1/2}\times\mu_J)$.
 Note that this shows that all measures in $\mathcal{M}_S(\Omega_3\times\Omega_2)$ generated by $(y,\chi_J)$ along $(n_k)_{k\in\N}$ are pushed forward by $\psi$ onto $\mu''$.
 Therefore to finish the proof of the Claim \ref{p.l.entropy} it is enough to see that
\begin{equation}
\label{e.toseethat}
h_{\mu''}(S) \ge \dund(J)\cdot \log 2.
\end{equation}
Let $\mathbb{I}_{[1]_0}$ be the characteristic function of the cylinder $[1]_0\subset\Omega_2$.
Note that from the definition of   $\dund(J)$ and $\dbar(J)$ it follows immediately that
\begin{equation}
\label{e.dj}
\dund(J) \le
\lim_{k\to\infty} \frac{1}{n_k}\sum_{j=0}^{n_k-1} \mathbb{I}_{[1]_0}(\sigma^{j}(\chi_J))
\le  \dbar (J).
\end{equation}
Observe also that the measure $\mu''$, by its definition,  is concentrated on the set
\[
\big( [0]_0\times[1]_0 \big) \cup  \big( [1]_0 \times [1]_0 \big) \cup \big( [2]_0\times [0]_0 \big)\subset \Omega_3\times\Omega_2.
\]
Consider  the set $E\eqdef \big([0]_0\times[1]_0\big)\cup\big([1]_0\times [1]_0\big)\subset \Omega_3\times\Omega_2$
and let $\mathbb{I}_E$ be its characteristic function.
\begin{cla}
\label{cl.theclaim} We have
$0<\dund(J)\le\mu''(E)=\mu_J([1]_0)\le\dbar(J)$.
\end{cla}

\begin{proof}[Proof of Claim \ref{cl.theclaim}]
Note that for $n\in\N$ it holds $S^n(z,\chi_J)\in E$ if and only if $\sigma^n(\chi_J)\in [1]_0$, equivalently, if $n\in J$.
Recall that along $(n_k)$, the point $(z,\chi_J)$ generates $\mu''$ and $\chi_J$ generates $\mu_J$.
Furthermore,
as the topological boundaries of $E$ and $[1]_0$ are empty, it follows from the portmanteau theorem \cite[Thm. 18.3.4]{Garling}, that
\[
\mu''(E)=\lim_{k\to\infty}\frac{1}{n_k}\sum_{j=0}^{n_k-1} \mathbb{I}_E(S^{j}(z,\chi_J))
=\lim_{k\to\infty} \frac{1}{n_k}\sum_{j=0}^{n_k-1} \mathbb{I}_{[1]_0}(\sigma^{j}(\chi_J)).
\]
Equation \eqref{e.dj} implies now that
$0<\dund(J)\le\mu''(E)\le\dbar(J)$, proving Claim \ref{cl.theclaim}.
\end{proof}

By Claim \ref{cl.theclaim} we have that $S$ induces a measure preserving system $(E,\mu''_E,S_E)$ on $E$,
where $\mu''_E(A)=\mu''(A\cap E)/\mu''(E)$ for every Borel set $A\subset \Omega_3\times\Omega_2$
and $S_E(x) =S^{r(x)} (x)$, where $r(x)=\inf\{ q>0 \colon S^q(x) \in E\}$ is defined for $\mu''$-a.e. point $x\in E$.

\begin{cla}\label{l.bernoulli}
The measure preserving system $(\Omega_2,\xi_{1/2},\sigma)$ is a factor of $(E,\mu''_E,S_E)$.
\end{cla}

Let us assume that Claim \ref{l.bernoulli} holds and conclude the proof of Claim~\ref{p.l.entropy}.
Note that by Claim \ref{l.bernoulli} we have
$h_{\mu''_E}(S_E)\ge \log 2$. Now by Abramov's formula\footnote{The proof that this well-known formula works for transformations which can be not ergodic nor invertible, is due  H. Scheller, see \cite{Krengel} or \cite[p. 257]{Petersen}.}
it follows that
\[
h_{\mu''_E}(S_E)=h_{\mu''}(S)/\mu''(E).
\]
By Claim~\ref{cl.theclaim}, this yields $\dund(J)\cdot \log 2\le h_{\mu''}(S)$,
proving \eqref{e.toseethat} and finishing the proof of Claim~\ref{p.l.entropy}.
\end{proof}
Since Claim~\ref{p.l.entropy} implies Claim \ref{thm:mainbis}, and the latter implies Theorem \ref{t.p.entropy},
it remains to prove Claim \ref{l.bernoulli}.
\begin{proof}[Proof of Claim~\ref{l.bernoulli}]
 Consider the partition $\cP_E\eqdef \{P_0,P_1\}$ of $E$, where
$P_0\eqdef [0]_0\times[1]_0$ and $P_1\eqdef [1]_0\times [1]_0$. 
Fix $N\in\N$ and  $v=v_1\ldots v_N\in\{0,1\}^N$. Let
\[
\mathcal P_v\eqdef  P_{v_1}\cap S_E^{-1}(P_{v_2})\cap\ldots\cap S_E^{-N+1}(P_{v_N}).
\]
Our goal is to prove that $\mu''_E(\mathcal P_v) = 1/2^N$, which implies that $(\Omega_2,\xi_{1/2},\sigma)$ is a factor $(E,\mu''_E,S_E)$ through the factor map generated by $\cP_E$.
To this end we need some auxiliary notation. Let $\mathcal G_J^N$ be the set of blocks over $\{0,1\}$ which contain exactly $N$ occurrences of $1$, start with $1$, and end with $1$. For $u\in\mathcal G_J^N$ and $1\le j\le N$ we denote by $o(j)$ the position of the $j$-th occurrence of $1$ in $u$
and define
\[
V_{v,u}\eqdef \{(\omega,\bar\omega)\in \supp\mu'': \bar \omega\in [u]_0 \text{ and }\omega_{o(j)}=v_{j}\text { for }j=1,\ldots,N\}.
\]
From the definition of $\mu''$ it follows that
$\mu''(V_{v,u})=(1/2^N)\,\mu_J([u]_0)$. 
Furthermore, we set
\[
U_v\eqdef \bigcup_{u\in \mathcal G_J^N}V_{v,u},
\qquad\text{hence}\qquad
\mu''(U_v)=\frac{1}{2^N}\, \sum_{u\in \mathcal G_J^N}\mu_J([u]_0).
\]
Noting that
 $\mu_J$-almost every point $\bar\omega\in[1]_0$ belongs to some
$[u]_0$ with $u\in\mathcal G_J^N$
we get that
$$
\sum_{u\in \mathcal G_J^N}\mu_J([u]_0)=\mu_J([1]_0).
$$
 Therefore, using Claim~\ref{cl.theclaim},  we get
 $$
 \mu''(U_v)=\frac{\mu_J([1]_0)}{2^N}= \frac{\mu''(E)}{2^N}.
 $$
Note also that $\mathcal P_v=U_v\cap E=U_v$, thus
\[
\mu''_E(\mathcal P_v)=\frac{\mu''(U_v)}{\mu''(E)}= \frac{1}{2^N},
\]
proving Claim \ref{l.bernoulli}.
\end{proof}

The proof of Claim~\ref{t.p.entropy} is now complete. This ends the proof of Theorem \ref{thm:mainbis}.
\end{proof}

\subsection{Rational sets and proof of Theorem~\ref{thm:main}}\label{ss.rational}
The notion of a rational set was introduced by Bergelson and Ruzsa \cite{BR1}.
Below,  by an {\emph{arithmetic progression}} we mean a set of the form $a\Z+b$ for some $a,b\in\N$, $a\neq 0$.
\begin{defn}\label{d.rational}
We say that a set $A\subset\Z$ is \emph{rational} if for every $\eps>0$ there is a set $B\subset\Z$ which is a union of finitely many arithmetic progressions and satisfies $\dbar(A\symdiff B)<\eps$, where $A\symdiff B$ stands for the symmetric difference of $A$ and $B$. A subset $B$ of $\N$ is \emph{rational} if $B=C\cap \N$ for some rational set $C\subset \Z$.
\end{defn}

\begin{rem}\label{r.complement}
 If $A \subset \Z$ is rational then its complement $\Z\setminus A$ is also rational. The same holds for rational subsets of $\N$.
 Note that, by definition, a rational set has a well defined density.
 \end{rem}

Recall here that the formula
\[
\dbar(x,y)\eqdef \limsup_{N\to\infty}\frac{1}{N}
\card{\{0\le n<N:x_n\neq y_n\}},\quad \text{for }x,y\in\Omega_M
\]
defines a pseudometric on $\Omega_M$.
In the following, we need the following properties  of  $\dbar$:
\begin{enumerate}
  \item\label{p1} If $(x_n)_{n\in\N}\subset\Omega_M$ is a sequence of ergodic points and $x\in\Omega_M$ is such that $\dbar(x_n,x)\to 0$ as $n\to \infty$,
then $x$ is also an ergodic point.
  \item\label{p2} Furthermore, if $(x_n)_{n\in\N}\subset\Omega_M$ and $x\in\Omega_M$ are as above and $V(x_n)=\{\mu_n\}$ and $V(x)=\{\mu\}$, then $h_{\mu_n}(\sigma)\to h_\mu(\sigma)$ as $n\to\infty$.
\end{enumerate}
A proof of \eqref{p1} is sketched in \cite{Weiss}, alternatively it follows from \cite[Theorem 15 and Corollary 5]{KLO}.
To see \eqref{p2} one combines \eqref{p1} with
\cite[Theorem I.9.16]{Shields} and the proof of \cite[Theorem I.9.10]{Shields}.
Corollary~\ref{c:main} (and therefore Theorem~\ref{thm:main}) follows from the following result.

\begin{lem}\label{lem:genericity}
If $A\subset \Z$ or $A\subset \N$ is a rational set, then its characteristic function $\chi_A\in\Omega_2$ is a completely deterministic ergodic point. 
\end{lem}
\begin{proof}
Note that a set $B\subset \Z$ is a union of finitely many arithmetic progressions  if and only if its characteristic function
$\chi_B\in\Omega_2$ is a periodic point for the shift map $\sigma\colon\Omega_2\to\Omega_2$. Furthermore,  
$\dbar(A\symdiff B)<\eps$
is equivalent to $\dbar(\chi_A,\chi_B)<\eps$.
Thus, by definition, for every rational set $A\subset \Z$ there is a sequence $(z_n)_{n=0}^\infty$ of $\sigma$-periodic points in $\Omega_2$ such that $\dbar(\chi_A,z_n)\to 0$ as $n\to\infty$.
Since each $z_n$ is a generic point for a zero entropy ergodic measure, the properties \eqref{p1}--\eqref{p2} of $\dbar$ mentioned above allow us to finish the proof if $A\subset \Z$. For $A\subset\N$ it is enough to note that if $C\subset \Z$ is such that $\chi_C\in\Omega_2$ is a completely deterministic ergodic point and $A=C\cap\N$, then $\chi_A\in\Omega_2$ is also a completely deterministic ergodic points, as these notions depend only on the forward orbit of a point.
\end{proof}

Using Lemma~\ref{lem:genericity} we see that the characteristic function $\chi_A\in\Omega_2$
of a rational set $A\subset \Z$ is a completely deterministic ergodic point and has a well defined density. Therefore we get the following corollary of Theorem \ref{thm:mainbis}, which explains why the theorem about control of entropy stated in the introduction follows from Theorem \ref{thm:mainbis}.
\begin{coro}\label{c:main}
The conclusion of Theorem~\ref{thm:mainbis} holds if $J$ is a rational set of $\N$ with $0<d(J)<1$.
\end{coro}






\section{Rational  long  sparse tails}
\label{s.scalesandtails}

The aim of this section is to find rational subsets of $\N$, which fulfill the requirements needed in the construction of a controlled point from \cite{BDB:}. These sets are coined in \cite{BDB:} \emph{$\cT$-long $\bar\varepsilon$-sparse tails} (associated to a scale $\cT$ and a controlling sequence $\bar\varepsilon$) and are discussed below under the name of \emph{$\cT$-sparsely long tails}. Their elements are times where
the control of the averages of a function is partially lost.  The times in a tail allows us to control and
spread the support of measures generated by a controlled point (here longness of a tail is crucial for guaranteeing the full support), while we retain some control on the averages due to sparseness. To control the entropy we define our tail so that it is also a rational set. Hence its complement, that is, the set of times defining the regular part of the orbit, is also rational. It follows that both sets, the tail and the regular part, have nontrivial and well-defined densities. This allows us to apply the criterion for positivity of the entropy (Theorem \ref{thm:mainbis}) to the measures generated by a controlled point.

We  first recall  from \cite{BDB:} the definitions
of scales and sparsely long tails.

\begin{defn}[Scale]\label{d.scale}
We say that a sequence of positive integers $\cT=(T_n)_{n\in\N}$ is \emph{a scale} if
  there is an integer sequence  $\bar \kappa=(\kappa_n)_{n\in\N}$ of \emph{factors of $\cT$}
  such  that
 \begin{itemize}
  \item $\kappa_0=3$, and $\kappa_{n}$ is a multiple of $3\kappa_{n-1}$ for every $n\ge 1$,
  \item
  $T_0$ is a multiple of $3$ and
   $T_{n}=\kappa_{n} \, T_{n-1}$ for every $n\ge 1$,
  \item
  $\kappa_{n+1}/\kappa_n \to \infty$ as $n\to\infty$.
\end{itemize}
\end{defn}
\begin{rem}
We have that
\[
\sum_{n=0}^\infty\frac{1}{\kappa_n}\le\sum\frac{1}{3^n}<1.
\]  
\end{rem}
\begin{defn}[$\N$-interval]
An \emph{interval of integers} (or $\N$-interval for short) is a set
$[a,b]_{\NN}\eqdef[a,b]\cap \NN$, where $a,b\in \NN$.
\end{defn}
\begin{defn}[Component of a set $\MM\subset \N$]
Given a subset $\MM$ of $\NN$
a \emph{component} of $\MM$ is any maximal $\N$-interval contained in $\MM$, that is,
$[a,b]_{\NN}\subset \MM$ is a component of $\MM$ if and only if $b+1\notin \MM$ and $a-1\notin \MM$.
\end{defn}
\begin{defn}[$T$-regular interval]
Let $T$ be a positive integer. We say that an $\N$-interval $I$ is \emph{$T$-regular interval}  if
$I=[kT,(k+1)T-1]_\N$ for some $k\ge 0$.
\end{defn}

\begin{defn}[$\cT$-adapted set, $n$-skeleton]
Let $\cT=(T_n)_{n\in \NN}$ be a scale. We say that a set $R_\infty\subset\N$ is \emph{$\cT$-adapted} if every component of $R_\infty$ is a $T_n$-regular interval for some $n\in\N$.
Given $n\in \N$ the \emph{$n$-skeleton} $R_n$ of $R_\infty$ is the union of all components of $R_\infty$ which are $T_k$-regular intervals for some $k\ge n$.
\end{defn}

By definition for any $\cT$-adapted set $R_\infty$ we have $R_\infty=R_0\supset R_1\supset R_2\supset\ldots$.


%
%

\begin{defn}[Sparsely long tail]\label{d.tail}
Consider a scale
$\cT=(T_n)_{n\in \NN}$ with a sequence of factors $\bar\kappa=(\kappa_n)_{n\in \NN}$.
A  set $R_\infty\subset \NN$ is a \emph{$\cT$-sparsely long tail} if the following holds:
\begin{enumerate}
\item\label{i.adapted}
$R_\infty$ is $\cT$-adapted,
\item\label{i.0} $0\notin R_\infty$, in particular
$[0,T_n-1]_{\NN} \not\subset R_n$, where $R_n$ is the $n$-skeleton of $R_\infty$,
\item\label{i.center} If a $T_n$-regular interval $I=[a,b]_\N$ is not contained in $R_\infty$,
equivalently if $I\not\subset R_n$, then the $(n-1)$-skeleton $R_{n-1}$ of $R_\infty$ can intersect nontrivially only in the middle third interval of $I$ and is {\emph{$1/\kappa_n$-sparse}} in $I$, that is
\begin{gather*}
I\cap R_{n-1}\subset\left[ \left(a +\nicefrac{T_n}{3}\right),\left(b-\nicefrac{T_n}{3}\right)\right]_\NN,   \\
0< \frac{ \card{R_{n-1}\cap I}}{T_n}\le 1/\kappa_n.
\end{gather*}
\end{enumerate}
\end{defn}

\begin{defn}[Rational sparsely long tail]
We say that $R_\infty$ is a \emph{rational $\mathcal{T}$-sparsely long tail} if it satisfies Definitions \ref{d.rational} and \ref{d.tail}.
\end{defn}

Next, we extend \cite[Lemma 2.7]{BDB:} adding rationality of the tail to its conclusion. In fact, the tail constructed in \cite{BDB:} is also a rational tail but this fact is not noted there.


\begin{propo}[Existence of rational sparsely long tails]\label{p.l.tailexistence}
Let $\cT=(T_n)_{n\in \NN}$ be a scale and $\bar \kappa=(\kappa_n)_{n\in\N}$ be its  sequence of factors.
Then there is a rational $\cT$-sparsely long tail $R_\infty$ with $0<d(R_\infty)<1$.
\end{propo}

\begin{proof}
For each $n\in\N$ define  $A_{n+1} \eqdef\left[\nicefrac{T_{n+1}}{3}, \nicefrac{T_{n+1}}{3}+T_n-1 \right]_\NN$. Then we set
\[
R^*_\infty = \bigcup_{n\ge 1}  (A_n + T_n \ZZ)\qquad\text{and}\qquad R_\infty=R^*_\infty\cap \N.
\]
It is easy to see that the requirements imposed on the growth of the scale $\cT$ imply that the characteristic function $\chi_\infty$ of $R^*_\infty$ is a \emph{regular Toeplitz sequence} (see \cite{Downarowicz-Toeplitz}), which immediately yields that $R_\infty$ is rational. Clearly, $R_\infty\neq \N$, which gives  $0<d(R_\infty)<1$, again because $\chi_\infty$ is Toeplitz. But for the convenience of the readers we provide a direct elementary proof of these facts.

Let $\Pi_0=\emptyset$ and for each $n\ge1$ we set $\Pi_n=R_\infty\cap[0,T_n-1]_\N$. It follows that for each $n\in\N$ we have $\Pi_n\subset\Pi_{n+1}$ and
\begin{equation}\label{e.pin}
\Pi_{n+1}=
\left( \bigcup_{k=0}^{\kappa_{n+1}-1}  \Pi_n + k T_n  \right) \cup A_{n+1}.
\end{equation}
\begin{cla}
\label{cl.longsparse}
 The set $R_\infty $ is a $\cT$-sparsely long tail.
\end{cla}

\begin{proof}We need to check conditions \eqref{i.adapted}, \eqref{i.0} and \eqref{i.center} from Definition \ref{d.tail}.
First we prove \eqref{i.adapted}, which says that $R_\infty$ is $\mathcal{T}$-adapted. Fix $n\ge 0$ and note that neither $0$, nor $T_{n+1}-1$ belongs to $\Pi_{n+1}$. We claim that the components of $\Pi_{n+1}$ are $T_i$-regular for some $i\le n$. Note that the $T_n$-regular interval $A_{n+1}$ is a component of $\Pi_{n+1}$. The other components of $\Pi_{n+1}$ are components of $\Pi_n$  translated by a number $\ell T_n$, for some $\ell\in \{0,\dots,\kappa_{n+1}-1\}$. Arguing inductively, we get that the components of $R_\infty$ are $T_i$-regular for some $i\in\N$.
Obviously $0\not \in R_\infty$, yielding \eqref{i.0}. It remains to prove \eqref{i.center}. Fix $n\ge 1$. Note that
$
R_{n-1} \cap \Pi_n= A_n,
$ and $A_n$ is contained in the middle third interval of $\Pi_n$. 
It follows that
$$
\frac{\card{R_{n-1} \cap \Pi_n}}{T_n}= \frac{\card{A_n}}{T_n}=\frac{T_{n-1}}{T_n}=\frac{1}{\kappa_n}.
$$
This proves that $T_n$-regular interval $[0,T_n-1]\not\subset R_\infty$ containing $\Pi_n$ satisfies \eqref{i.center}.
The same holds for every $T_n$-regular interval $I$ not contained in $R_\infty$, because such $I\cap R_\infty=\Pi_n+j T_n$ for some $j\ge 1$.
\end{proof}

\begin{cla}
\label{cl.rational}
 The set $R_\infty $ is rational.
 \end{cla}
\begin{proof}
Define $Q_n\eqdef  \bigcup_{i=1}^{n} (A_i + T_i \N)$. Then $Q_n$is a finite union of arithmetic sequences. Furthermore, $R_\infty\symdiff Q_n  \subset R_n$ and the $n$ skeleton $R_n$ of
$R_\infty$ satisfies
\begin{equation}\label{e.rn}
R_n =\bigcup_{i=n+1}^\infty (A_i + T_i \N).
\end{equation}
Thus it is enough to see that $\bar d (R_n) \to 0$ as $n\to\infty$. But by \eqref{e.rn} and subadditivity of $\dbar$ we have
\[
\dbar(R_n) =\dbar\bigg(\bigcup_{i=n+1}^\infty (A_i + T_i \N)\bigg)\le \sum_{i=n+1}^\infty \dbar(A_i + T_i \N).
\]
It is easy to see that for each $i\ge 1$ we have $\dbar(A_i + T_i \N) = T_{i-1}/T_i=1/\kappa_i$. As a conclusion, we get
$\dbar (R_n) \le \sum_{i=n+1}^\infty \frac{1}{\kappa_n}$,
 proving that the tail is rational.
\end{proof}

\begin{cla}\label{cl.density}
We have $0<d(R_\infty)<1$.
 \end{cla}

\begin{proof}
The set is rational and hence
it has a well defined density, see Remark~\ref{r.complement}. We also have
\[
\frac{1}{\kappa_1}=
\dbar(A_1+T_1\N)\le \dbar(R_\infty)=\dbar\bigg(\bigcup_{i=1}^\infty (A_i + T_i \N)\bigg)
\le\sum_{i=1}^\infty\frac{1}{\kappa_i}<1.\qedhere
\]
\end{proof}

The lemma now follows from Claims~\ref{cl.longsparse}, \ref{cl.rational},
and
\ref{cl.density}.
\end{proof}



\section{Double flip-flop families}
\label{s.flipflop}

In this section, we review the definitions of flip-flop families  following \cite{BBD:16,BDB:}. We also introduce the notion of a double flip-flop family and
prove that flip-flops families yield double flip-flop families, see Proposition~\ref{p.yieldsdouble}.
Using these families and Theorem~\ref{thm:main} obtain ergodic measures  with full support, positive entropy, and zero average for a continuous potential $\varphi\colon X \to \mathbb{R}$, see Theorem~\ref{t.flipfloptailqual}.

In what follows, $(X, \rho)$ is a compact metric space,
$f\colon X \to X$ is a homeomorphism, and $\varphi\colon X \to \mathbb{R}$ is
a continuous function.

\subsection{Flip-flop families}
 We begin by recalling the definition of flip-flop families.

\begin{defn}[Flip-flop family]\label{d.flipflop}
A \emph{flip-flop family} associated to $\varphi$ and $f$ is a family
$\fF=\fF^+\sqcup\fF^-$ of compact subsets of $X$ called {\emph{plaques}}\footnote{We pay special attention to the case when the sets of the flip-flop family
are discs tangent to a strong unstable cone field. This justifies this name.}
such that there are $\alpha>0$ and a
sequence of numbers  $(\zeta_n)_n$ and $\zeta_n\to 0^+$ as $n\to \infty$,
satisfying:
\begin{enumerate}
 \item\label{i.flipflop1}
 let $F_\fF^+\eqdef  \bigcup_{D\in\fF^+} D$ (resp. $F_\fF^-\eqdef \bigcup_{D\in\fF^-} D $), then  $\varphi(x)>\alpha$ for every $x\in F_\fF^+$
 (resp. $\varphi(x)< -\alpha$ for every $x\in F_\fF^-$);
 \item\label{i.flipflop2}
 for every $D\in \fF$, there are sets $D^+\in \fF^+$ and $D^-\in\fF^-$ contained in $f(D)$;
 \item\label{i.flipflop33}
 for every  $n>0$ and  every family of sets $D_i\in \fF$, $i\in\{0,\dots ,n\}$,  with $D_{i+1}\subset f(D_i)$
 it holds
  $$
   \mathrm{diam} (f^{-i} (D_n))\le \zeta_i, \quad \mbox{for every $i\in\{0,\dots,n\}$.}
 $$
\end{enumerate}
\end{defn}

We now recall the notions of $f$-sojourns. Note that in this definition the flip-flop family may be relative to
a power $f^k$ of $f$, but the sojourns are relative to $f$. Furthermore, $\varphi_k$ stands here for the Birkhoff averages of a map $\varphi$ with respect to $f$ introduced in \eqref{e.averages}.

\begin{defn}[Flip-flop family with $f$-sojourns]\label{d.flipfloptail}
Consider a  flip-flop  family
 $\fF=\fF^+\sqcup\fF^-$
  associated to $\varphi_k$ and $f^k$ for some $k\ge 1$ and
a compact subset $Y$ of $X$.
The \emph{flip-flop family $\fF$ has $f$-sojourns along $Y$}
(or   \emph{$\fF$ $f$-sojourns along $Y$}) if
there is a sequence $(\eta_n)_n$ and $\eta_n\to 0^+$, such that
for every $\delta>0$ there is an integer $N=N_\delta$ so that
every plaque $D\in\fF$ contains subsets $\widehat D^+, \widehat D^-$ satisfying:
\begin{enumerate}
 \item\label{i.defff0}
 for every $x\in \widehat D^+\cup \widehat D^-$ the orbit segment $\{x,\dots, f^N(x)\}$ is $\delta$-dense in $Y$
 (i.e., the $\delta$-neighbourhood of the orbit segment contains $Y$);
 \item\label{i.defff1}
 $f^N(\widehat D^+)=\widehat D^+_N\in \fF^+$ and $f^N(\widehat D^-)=\widehat D^-_N\in\fF^-$;
 \item\label{i.defff2}
 for every $i\in\{0,\dots, N\}$
 it holds
 $$
   \mathrm{diam} (f^{-i} (\widehat D_N^\pm))\le \eta_i.
 $$
\end{enumerate}
\end{defn}
 We are now in position to define double flip-flop families (with sojourns). Observe that the remark before Definition \ref{d.flipfloptail} applies also to Definition \ref{d.dflipfloptail}.

\begin{defn}[Double flip-flop family]\label{d.dflipflop}
A \emph{double flip-flop family} associated to $\varphi$ and $f$ is a family
$\fD\eqdef \fD^+_0\sqcup\fD^+_1 \sqcup \fD^-_0\sqcup\fD^-_1 $ of compact subsets of $X$ such that there are $\alpha>0$ and a
sequence of numbers  $(\zeta_n)_n$ with $\zeta_n\to 0^+$ as $n\to \infty$,
and with the following properties:
let
$E^+_i\eqdef  \bigcup_{D\in\fD^+_i} D$ and
$E^-_i\eqdef  \bigcup_{D\in\fD^-_i} D$, where $i=0,1$,
$E^+\eqdef E^+_0 \cup E^+_1$, and
$E^-\eqdef E^-_0 \cup E^-_1$.
\begin{enumerate}
 \item\label{i.dflipflop1} $\varphi(x)\geq\alpha$ for every $x\in E^+$ and $\varphi(x)\leq -\alpha$ for every $x\in E^-$;
  \item\label{i.dflipflop2}
 for every $D\in \fD$, there are sets $D^+_0\in \fD^+_0$, $D^+_1\in \fD^+_1$, $D^-_0\in \fD^-_0$, and
 $D^-_1\in \fD^-_1$    contained in $f(D)$;
 \item\label{i.dflipflop3}
 for every  $n>0$ and  every family of sets $D_i\in \fD$, $i\in\{0,\dots ,n\}$,  with $D_{i+1}\subset f(D_i)$
 it holds
  $$
   \mathrm{diam} (f^{-i} (D_n))\le \zeta_i, \quad \mbox{for every $i\in\{0,\dots,n\}$.}
%
%
 $$
 \item\label{i.dflipflop4}
 The closures of the sets $E^+_0, E^+_1, E^-_0, E^-_1$ are pairwise disjoint\footnote{This condition is straightforward in the flip-flop case since
 $\varphi$ is strictly bigger that $\alpha>0$ in $F^+_\cF$ and strictly less than $-\alpha<0$ in $F^-_\fF$.}.
\end{enumerate}
\end{defn}

\begin{defn}[Double flip-flop family with $f$-sojourns along $Y$]\label{d.dflipfloptail}Let $\fD\eqdef \fD^+_0\sqcup\fD^+_1 \sqcup \fD^-_0\sqcup\fD^-_1 $ be a double flip-flop family associated to $\varphi_k$, $f^k$,  $k\ge 1$.
Given a compact subset $Y$ of $X$ we say that
\emph{$\fD$ has $f$-sojourns along $Y$}
(or that  \emph{$\fD$ $f$-sojourns along $Y$}) if
there is a sequence $(\eta_n)_n$ and $\eta_n\to 0^+$, such that
for every $\delta>0$ there is an integer $N=N_\delta$ such that
every plaque $D\in\fD$ contains subsets $\widehat D^+_0, \widehat D^+_1,
\widehat D^-_0, \widehat D^-_1$ such that:
\begin{enumerate}
 \item\label{i.def-dff0}
 for every $x\in \widehat D^+_0\cup \widehat D^+_1 \widehat D^-_0 \cup \widehat D^-_1$ the orbit segment $\{x,f(x),\dots, f^N(x)\}$ is $\delta$-dense in $Y$;
 \item\label{i.def-dff1}
 $f^N(\widehat D^i_j)= D^i_{N,j} \in \fD^i_j$, $i\in \{-,+\}$ and $j\in \{0,1\}$;
 \item\label{i.def-dff2}
 for every $i\in\{0,\dots, N\}$
 it holds
 $$
   \mathrm{diam} (f^{-i} (\widehat D_{N,j}^\pm))\le \eta_i.
 $$
\end{enumerate}
\end{defn}

\begin{rem} \label{r.multiple}
In the previous definitions the constant $N$ can be chosen a multiple of $k$.
\end{rem}

\subsection{Existence of double flip-flop families with sojourns}
\label{ss.double}
We now prove that existence of flip-flop families with sojourns implies the existence of  double flip-flop families with sojourns.

\begin{propo}
\label{p.yieldsdouble}
Consider a  flip-flop  family
 $\fF=\fF^+\sqcup\fF^-$
  associated to $\varphi_k$ and $f^k$  for some $k\ge 1$ with $f$-sojourns in
a compact subset $Y$ of $X$. Then there are $r\ge 1$ and a double flip-flop family $\fD= \fD^+_0\sqcup \fD^+_1 \sqcup
\fD^-_0\sqcup \fD^-_1$ associated to $f^r$ and $\varphi_r$ with $f$-sojourns along $Y$.
\end{propo}

\begin{proof}
Given a plaque $D\in \fF$ and $\ell \ge 1$, consider subsets
$D_{+^\ell,+}$ $D_{+^\ell,-}$, $D_{-^\ell,+}$, and $D_{-^\ell,-}$ of $D$ satisfying
\begin{itemize}
\item
$f^{ki} (D_{+^\ell,+})$ is contained in some plaque of $\fF^+$
for every $i\in \{1,\dots,\ell\}$ and  $f^{k(\ell+1)} (D_{+^\ell,+})\in \fF^+$,
\item
$f^{ki} (D_{+^\ell,-})$
is contained in some plaque of $\fF^+$
for every $i\in \{1,\dots,\ell\}$ and  $f^{k(\ell+1)} (D_{+^\ell,-})\in \fF^-$,
\item
$f^{ki} (D_{-^\ell,-})$ is contained in some plaque of $\fF^-$ for every $i\in \{1,\dots,\ell\}$ and  $f^{k(\ell+1)} (D_{+^\ell,-})\in \fF^-$,
\item
$f^{ki} (D_{-^\ell,+}))$ is contained in some plaque of $\fF^-$ for every $i\in \{1,\dots,\ell\}$ and  $f^{k(\ell+1)} (D_{-^\ell,+})\in \fF^+$.
\end{itemize}
The existence of these subsets is assured by item \eqref{i.flipflop2} in the  definition of a flip-flop family.

Using the continuity of $\varphi$, we have that for every  $\ell$ large enough there is $\alpha'>$ such that
\[
\begin{split}
&\varphi_{k(\ell+1)} (x) > \alpha'>0 \quad \mbox{if $x\in D_{+^\ell,\pm}$},\\
&\varphi_{k(\ell+1)} (x) < -\alpha'<0 \quad \mbox{if $x\in D_{-^\ell,\pm}$}.
\end{split}
\]
We use here that the $\ell+1$ Birkhoff averages of $\varphi_k$ with respect to $f^k$ are the same as $\varphi_{k(\ell+1)}$, that is, $k(\ell+1)$ averages of $\varphi$ with respect to $f$.
We fix such a large $\ell$ and define
\[
\begin{split}
\fD^+_0&\eqdef \{D_{+^\ell,+}, \, D\in \fF\}, \quad
\fD^+_1\eqdef \{D_{+^\ell,-}, \, D\in \fF\},\\
\fD^-_0&\eqdef \{D_{-^\ell,+}, \, D\in \fF\}, \quad
\fD^-_1\eqdef \{D_{-^\ell,-}, \, D\in \fF\}.
\end{split}
\]
By construction $\fD=\fD^+_0\sqcup \fD^+_1\sqcup \fD^-_0\sqcup \fD^-_1$ satisfies conditions
\eqref{i.dflipflop1}, \eqref{i.dflipflop2}, and \eqref{i.dflipflop3} in the definition of double flip-flop family for $f^{k(\ell+1)}$ and  $\varphi_{k(\ell+1)}$.
To check condition \eqref{i.dflipflop4},  i.e, the closures of the sets $E^+_0, E^+_1, E^-_0, E^-_1$  are pairwise disjoint, just observe that
the value of $\varphi$ on the $k\ell$ and  $k(\ell+1)$ iterates of theses sets are uniformly separated.

It remains to get the sojourns property.
Fix small $\delta>0$ and consider the number $N=N_\delta$ in the definition of sojourn for $\fF$.
Take a set $D\in \fD$ and consider $f^{k(\ell+1)}(D)=\widehat D\in \fF$. The sojourns property for $\fF$ provides a
subset $\widehat D'$ such that $f^N (\widehat D')\in \fF$ and the  first $N$ iterates of any point $x\in \widehat D'$
are $\delta$-dense in $Y$. Consider now $f^{-k(\ell+1)} (\widehat D')\subset D$. It is enough now to observe that any point in that set is such that its first
$k(\ell+1)+N$ iterates are $\delta$-dense in $Y$. We omit the choice of the sequences $\zeta_i$ and $\eta_i$ in the previous construction. We finish the proof by taking $r=k(\ell+1)$.
\end{proof}
%
%

\subsection{Support, average, and entropy}
\label{ss.support}
We now obtain ergodic measures  with full support and positive entropy satisfying $\int \varphi d\mu=0$.

\begin{theo}\label{t.flipfloptailqual}
Let $(X,\rho)$ be a compact metric space,
$Y$ a compact subset of $X$, $f\colon X \to X$ a homeomorphism, and $\varphi\colon X \to \mathbb{R}$
a continuous function.
Assume that there is  a flip-flop family $\fF$  associated to $\varphi_k$ and $f^k$ for some $k\ge 1$ having
$f$-sojourns along $Y$.
Then there is an ergodic measure $\mu$ with positive entropy whose support contains $Y$ and such that $\int \varphi \,d\mu =0$.

%
\end{theo}

First note that by Proposition~\ref{p.yieldsdouble} we can assume that there is a double flip-flop family
$\fD=\fD^+_0\sqcup \fD^+_1\sqcup \fD^-_0\sqcup \fD^-_1$
relative to
$\varphi_r$, $f^r$, and some $r\ge 1$ with $f$-sojourns along $Y$.
We define the sets
$$
K_0\eqdef
\mathrm{closure} \left(
\bigcup_{D\in \fD^+_0\cup \fD^-_0} D \right)
\quad
\mbox{and}
\quad
K_1\eqdef
\mathrm{closure} \left(
\bigcup_{D\in \fD^+_1\cup \fD^-_1} D \right).
$$
Recall that the sets $K_0$ and $K_1$ are disjoint.
The pair $\mathbf{K}\eqdef (K_0, K_1)$ is the {\emph{division associated to $\fD$.}}

We need to recall some definitions from \cite{BDB:}.
Consider sequences $\bar \delta =(\delta_n)_{n\in \NN}$,
$\bar \alpha =(\alpha_n)_{n\in \NN}$, and $\bar \varepsilon =(\varepsilon_n)_{n\in \NN}$ of positive numbers converging to $0$ as $n\to \infty$.
Consider  a scale $\mathcal{T}=(T_n)_{n\in \NN}$ and a $\mathcal{T}$-long $\bar \varepsilon$-sparse tail $R_\infty$. 

\begin{defn}[$\bar \alpha$-control and $\bar\delta$-denseness]
A point $x\in X$ is
{\emph{$\bar \alpha$-controlled for $\varphi$  
with a tail $R_\infty$}} if for every $n\in\N$ and every $T_n$-regular interval $I$ that is not strictly contained in a component
of $R_\infty$ it holds
$$
\frac{1}{T_n} \sum_{j\in I} \varphi (f^j(x))\in [-\alpha_n, \alpha_n].
$$
The orbit of a point $x\in X$ is {\emph{$\bar\delta$-dense in $Y$ along the tail $R_\infty$ }} if for every component $I$ of $R_\infty$ of size $T_n$
the segment of orbit $\{f^j(x), j\in I\}$ is $\delta_n$-dense in $Y$.
\end{defn}

We are now ready to state the main technical step of the proof of Theorem~\ref{t.flipfloptailqual}.
This  is a reformulation of \cite[Theorem 2]{BDB:} with an additional control of the itine\-ra\-ries. This control leads to
positive entropy.
For the notion of a $\bf K$-itinerary 
of a point over a set see Section~\ref{s.measurespositive}.
\begin{propo}\label{p.maintech}
Let  $(X, \rho)$ be a compact metric space, $Y$ a compact subset of $X$, $f\colon X \to X$ be  a homeomorphism, and
 $\varphi\colon X \to \mathbb{R}$ be a continuous map. Assume that
there is a  double flip-flop  family
$\fD$ associated to $\varphi_r$, $f^r$ for some $r\ge 1$ with $f$-sojourns along $Y$. Let $\mathbf{K}=(K_0,K_1)$ the division of $\fD$.

Consider sequences $\bar \alpha =(\alpha_n)_{n\in \NN}$ and  $\bar \delta=(\delta_n)_{n\in \NN}$ of positive numbers converging to $0$ and
$\omega \in \Omega_2$. 
Then there are a scale $\mathcal{T}$ and a rational and $\mathcal{T}$-sparsely long tail $R_\infty$ such that:  for every plaque $D\in \fD$ there is a point $x\in D$ satisfying
\begin{enumerate}
\item\label{i.p.control}
the Birkhoff averages of $\varphi_r$ along the orbit of $x$ with respect to $f^r$ are $\bar\alpha$-controlled 
with the tail $R_\infty$,
\item \label{i.p.density}
the $f$-orbit of $x$ is $\bar\delta$-dense in $Y$ along $R_\infty$,
\item \label{i.p.itinerary} $\omega$ is the $\mathbf{K}$-itinerary of $x$ with respect to $f^r$ over $J\eqdef \NN \setminus R_\infty$.
\end{enumerate}
\end{propo}

\subsubsection{Proposition~\ref{p.maintech} implies Theorem~\ref{t.flipfloptailqual}}
We now deduce Theorem~\ref{t.flipfloptailqual}. Let $x$ be the point given by Proposition~\ref{p.maintech}
associated to $\omega\in \Omega_2$, which is a generic point for the Bernoulli measure $\xi_{1/2}$.
By \cite[Proposition 2.17]{BDB:},
if $\widetilde \mu$ is a accumulation point of the sequence of empirical measures $(\mu_n(x,f^r))_{n\in \NN}$ (recall \eqref{e.empiric}), then
for $\widetilde \mu$-almost every point $y$ it holds $\frac{1}{n} \sum_{i=0}^{n-1} \varphi_r (f^{ri}(y))=0$.
This implies that for every $\mu$ generated by $x$ for $f$
it also holds that
\begin{equation}\label{e.zero-av}
\frac{1}{n} \sum_{i=0}^{n-1} \varphi (f^{i}(y))=0, \quad
\mbox{for $ \mu$-almost every point $y$.}
\end{equation}
Moreover,
by \cite[Proposition 2.2]{BDB:} for every measure $\mu$ generated by $x$
one has the $f$-orbit of  $\mu$-almost every point $y$ is dense in $Y$.

Since $R_\infty$ is rational we have that its complement $J=\N\setminus R_\infty$ is also rational and has  a density
$d(J)>0$, see Remark~\ref{r.complement}.  By Theorem~\ref{thm:main}, every $f^r$-invariant measure $\widetilde \mu$ generated by $f^r$ along the orbit of $x$ satisfies $h(\widetilde\mu) > d(J)\log 2=\lambda>0$. Therefore, every $f$-invariant measure $\mu$ generated by $f$ along the orbit of $x$
has entropy at least $\lambda/r$. This implies that the ergodic decomposition of $\mu$ has some measure $\nu$ with full support, positive entropy,
and, by \eqref{e.zero-av}, $\int \varphi d\nu=0$. The proof of Theorem~\ref{t.flipfloptailqual} is now complete. \hfill $\square$

\subsection{Proof of Proposition~\ref{p.maintech}}
Fix sequences $\bar \alpha =(\alpha_n)_{n\in \NN}$ and  $\bar \delta=(\delta_n)_{n\in \NN}$ of positive reals converging to $0$. Take any $\omega\in\Omega_2$. 





Consider a double flip-flop $\fD=\fD^+_0 \sqcup \fD^+_1 \sqcup \fD^-_0 \sqcup \fD^-_1$ associated to $\varphi_r$ and $f^r$ for some $r\ge 1$ with sojourns along $Y$. Let $\mathbf{K}=(K_0,K_1)$ denote the associated division of $\fD$. 
We let $\fD^+\eqdef \fD^+_0 \sqcup \fD^+_1$ and $\fD^-\eqdef\fD^-_0 \sqcup \fD^-_1$ and note that $\fD^+\sqcup\fD^-$ is a flip-flop family for $f^r$. We also consider
$\fD_0\eqdef \fD^+_0 \sqcup \fD^-_0$ and    $\fD_1\eqdef \fD^+_1 \sqcup \fD^-_1$.

Following \cite{BDB:} we will use the induction on $n$ to construct a scale $\cT=(T_n)_{n\in\N}$ and a $\cT$-sparsely long tail $R_\infty$ (see Section \ref{s.scalesandtails}) such that there exists a point $x\in X$ satisfying conditions \eqref{i.p.control}, \eqref{i.p.density}, and \eqref{i.p.itinerary} of our proposition.

After $n$ steps of our induction we will have $T_0,\ldots, T_n$ and  
$\Pi_{n-1}=R_\infty\cap [0,T_n-1]$. 
Assume that all these objects are defined up to the index $n-1$. Note that no parameters beyond $n$ are required to check that some set $R\subset [0,T_n-1]$ satisfies the conditions from the definition of the $\cT$-sparsely long tail. Furthermore, knowing that $\Pi_{n-2}$ satisfies these conditions we can use translates of this set by a multiple of $T_{n-1}$ to get a set which we declare to be $\Pi_{n-1}=R_\infty\cap[0,T_{n}-1]$. The double flip-flop family is used as follows: the partition $\fD^+\sqcup \fD^-$ is used for controlling averages and the partition
$\fD_0 \sqcup \fD_1$ is used to follow a prescribed itinerary.


In the above situation, following the reasoning in \cite{BDB:} we obtain that there is an infinite set $\cS$ of multiples of $T_{n-1}$ such that
for every $S\in\cS$ and every $R\subset [0,S-1]$ following the rules of a tail (up to time $S$) and such that $R\cap[0,T_{n-1}-1]=\Pi_{n-1}$, given any $D\in \fD$
there is a family of plaques
$D_i\in \fD$,
$i\not\in R$,
 such that
\begin{itemize}
\item
for every $i,j\in [0,S-1]\setminus R$ with $j>i$ it holds $ D_j \subset f^{r(j-i)} (D_i)$, 
\item
for every
$x\in f^{-rS} (D_{S})$ the 
orbit segment $\{x,f^r(x),\dots,f^{rS}(x)\}$ is controlled for $\varphi_r$ with parameters $(\alpha_1,\dots,\alpha_n)$  and the  tail $R$, that is, 
for every $i\le n$ and every $T_i$-regular interval $I$ contained in $[0,S-1]_\N$ that is not contained in $R$ the average of $\varphi_r$ over $I$ is in $[-\alpha_i, -\alpha_i/2] \cup [\alpha_i/2, \alpha_i]$.

\item for every $x\in f^{-rS} (D_{S})$ and every component $I=[a,b]_\N$ of $R$ of size $T_i$
 the orbit segment $\{f^{i}(x):i\in [ar,br]_\N\}$ is $\delta_i$-dense in $Y$.
 \end{itemize}
Actually, exploiting the fact that we deal with double flip-flop family, we can combine the reasoning of \cite[Section 2.5.2]{BDB:} with the one in \cite{BBD:16} to add one more claim: we choose $D_i\in \fD_{\omega_i}$. 

Now, given $\cS$ we can choose $T_n$ which is large enough to obtain that $T_{n-1}/T_n$ is sufficiently small (since $\cS$is infinite we can do it). 
Furthermore we can extend $\Pi_{n-2}$ to $\Pi_{n-1}$ exactly as in the proof of Proposition \ref{p.l.tailexistence} see formula \eqref{e.pin}. 
This completes the induction step of our construction $R_\infty$.

Now observe that the set $\bigcap_{i\in R_\infty}f^{-ri}(D_i)\subset D_{0}$ is a nested intersection of nonempty compact sets with diameters converging to $0$, thus it contains only one point $x$, which by our construction satisfies conditions \eqref{i.p.control}, \eqref{i.p.density}, and \eqref{i.p.itinerary}.

\section{Robustly transitive diffeomorphisms}
\label{s.robustly}

In this section, we prove Theorems~\ref{t.openanddense} and \ref{t.average}. 

Recall that $\cR\cT(M)$ is the  (open) subset of $\diff^1(M)$
of diffeomorphisms that are
 robustly transitive,
 have a pair  of hyperbolic periodic points of different indices,
 and have
a partially hyperbolic  splitting
$TM = E^{\mruu} \oplus E^{\mrc} \oplus E^{\mrss}$
with one-dimensional center $E^{\mathrm{c}}$,
where
$E^\mathrm{uu}$ is uniformly expanding and
	$E^\mathrm{ss}$ is uniformly contracting.
	Let $d^\mathrm{uu}$ be the dimension of
	$E^\mathrm{uu}$.
	Note that  the map
$\mathrm{J}_f^{\mrc}	\colon M \to \RR$,
$\mathrm{J}_f^{\mrc}(x) \eqdef \log | Df_x |_{E^\mrc (x)}|$ is continuous for every $f\in\cR\cT(M)$.
Recall that $(\mathrm{J}_{f}^{\mrc})_n$ stands for the Birkhoff
$n$-average of $\mathrm{J}_f^{\mrc}$, cf. \eqref{e.averages}.

\begin{theo}\label{t.l.h(q)}
There is a $C^1$-open and dense subset $\cI(M)$ of $\cR\cT(M)$ such that
for every $f\in \cI(M)$ there are $N\in \NN$ and a neighbourhood $\mathcal \cU_f\subset \cI(M)$ such that
every $g\in \cU_f$ has
a flip-flop family with respect to the map
$(\mathrm{J}_g^{\mrc})_N$ and $g^N$   
 with sojourns along $M$.
\end{theo}

A crucial point here is that we get a single $N$ such that for every $g$ near $f$ we get a flip-flop family associated with
$g^N$ (we will pay special attention to this fact). A priori, the number $N$ for the flip-flop families constructed in \cite{BBD:16,BDB:} could depend on $f$. Since the flip-flop family for $f^N$ leads (through the criterion in Theorem \ref{thm:mainbis}) to an invariant measure with entropy bounded below by a constant times $\log 2/N$, we need the number $N$ to be locally invariable to get uniform local lower bounds for the entropy of  measures we find. This is precisely what we obtain from Theorem~\ref{t.l.h(q)}.

\begin{proof}[Sketch of the proof Theorem~\ref{t.l.h(q)}]
Our hypotheses imply that every
$f\in \cI(M)$ has a pair of saddles $p_f$ and $q_f$ of indices, respectively, $d^\mathrm{uu}$ and $d^\mathrm{uu}+1$. The saddles depend continuously on $f$  and the indices  are locally constant. Furthermore, the homoclinic classes of the saddles satisfy $H(p_f,f)=H(q_f,f)=M$ (see \cite[Proposition 7.1]{BDB:}
which just summarises results from \cite{BDPR}).

The discussion below involves the notions of a  {\emph{dynamical blender}} and a {\emph{flip-flop configuration.}}
As we do not need their precise definitions and will only use some specific properties of them,
we will just give rough definitions of these concepts  an refer to
\cite{BDB:} and \cite{BBD:16} for details.  In what follows, the discussion is restricted to our partially hyperbolic setting and to small open subset
of $\cI(M)$ where the index $d^\mathrm{uu}$ is constant.

%
Recall that a family of discs $\mathfrak{D}$  is \emph{strictly $f$-invariant} if there is
 an $\varepsilon$-neigh\-bour\-hood of  $\mathfrak{D}$
such that for every disc
$D_0$ in a such a  neighbourhood
 there is a disc $D_1\in \mathfrak{D}$ with
$D_1\subset f(D_0)$, see \cite[Definition 3.7]{BBD:16}.

 A {\emph{dynamical blender}} (in what follows we simply say a {\emph{blender}}) of a diffeomorphism $f$ is a locally maximal (in
an open set $U$) and transitive hyperbolic set
 $\Gamma$ of index $d^\mathrm{uu}+1$
endowed with an strictly $f$-invariant family of discs $\mathfrak{D}_f$ of  dimension $d^\mathrm{uu}$
tangent to  an invariant expanding cone field ${\mathcal{C}}^{\mathrm{uu}}$ around $E^\mathrm{uu}$. Hence, a blender is $4$-tuple $(\Gamma_f, U,{\mathcal{C}}^{\mathrm{uu}}, \mathfrak{D}_f)$. In what follows, let us simply denote the blender as $(\Gamma_f,\mathfrak{D}_f)$.

  As the usual hyperbolic sets,
blenders are $C^1$-robust  and have continuations. 
By \cite[Lemma 3.8]{BBD:16} strictly invariant families are robust: for every $g$ sufficiently close to $f$ the family $\mathfrak{D}_f$ is also strictly invariant for $g$.
As a consequence, if   $(\Gamma_f,\mathfrak{D}_f)$ is a blender of $f$ then $(\Gamma_g,\mathfrak{D}_f)$ is a blender of $g$ for every $g$ close to $f$,
where $\Gamma_g$ is the hyperbolic continuation of $\Gamma_f$. In what follows we will omit the subscripts for simplicity.

We can speak of the index of a blender $(\Gamma,\mathfrak{D})$ (the dimension of the unstable bundle of $\Ga$).
Given a saddle  of the same index as  the blender  we say that the blender  and the saddle are homoclinically related if their invariant manifolds intersect cyclically and transversely (this is a natural extension of the homoclinic relation of a pair of saddles).

In what follows, we consider blenders  which are expanding in
the center direction, that is, with index $d^\mathrm{uu}+1$.
Consider now a saddle $p$ of index $d^\mathrm{uu}$. The saddle $p$ and the blender $(\Gamma,\mathfrak{D})$ are in
a {\emph{flip-flop configuration}} if $W^\mathrm{u} (p,f)$ contains some disc of the family $\mathfrak{D}$  of the blender
 and the unstable manifold of the blender transversely intersects the stable manifold of the saddle
(note that the sums of these manifolds exceeds by one the dimension of the ambient space). By transversality and the openess of the invariant family the flip-flop configurations are also $C^1$-robust.

The results of \cite[Section 6.5.1]{BDB:} are summarised in the following proposition.
\begin{propo}
\label{p.dandovoltas}
There is an open and dense subset $\cF(M)$ of
$\cR\cT(M)$
such that every diffeomorphisms $f\in \cF(M)$ has a pair of saddles
$p_f$ and $q_f$ of different indices and
 a blender $(\Ga_f,\mathfrak{D}_f)$ such that:
 \begin{itemize}
 \item
 $H(p_f,f)=H(q_f,f)=M$,
 \item
 $\Gamma_f$ is
homoclinically related to $p_f$,
\item
$\Gamma_f$ and $q_f$
 are in a flip-flop configuration,
\item there is a metric on $M$ such that
$\mathrm{J}_f^{\mrc}$ is positive in a neighbourhood of $\Ga_f$ and negative in a neighbourhood
of the orbit of $q_f$.
\end{itemize}
\end{propo}

From now on, we will always consider $M$ with a metric given by Proposition \ref{p.dandovoltas}.
Let us recall another result from \cite{BDB:}.

\begin{theo}[Theorem 6.8 in \cite{BDB:}] \label{t.p.flipfloptail}
Consider $f\in \Diff^1(M)$ with a
dynamical blender $(\Ga,\mathfrak{D})$ in a flip-flop configuration with a hyperbolic periodic point $q$.
Let
$\varphi\colon M\to\RR$ be a continuous function such that $\varphi|_\Ga>0$ and $\varphi|_{\cO(q)}<0$.

Then there are $N\geq1$ and a flip-flop family $\mathfrak{F}$
 with respect to $\varphi_N$ and
 $f^N$  which  $f$-sojourns along the homoclinic class
 $H(q,f)$.
%
\end{theo}

We can now apply Theorem~\ref{t.p.flipfloptail} to the flip-flop configuration associated to the blender $(\Gamma_f,\mathfrak{D}_f)$
and the saddle $q_f$ provided by Proposition~\ref{p.dandovoltas} and the map $\mathrm{J}_f^{\mrc}$.
This provides the flip-flop family associated to the map $\mathrm{J}_f^{\mrc}$. The fact that the sojourns take
place in the whole manifold follows from $H(q_f,f)=M$.
To complete the sketch of the proof of
Theorem~\ref{t.l.h(q)} it remains to get the uniformity of $N$.

To get such a control we need to recall some steps of the construction in \cite{BBD:16}.

Let us explain how to derive the flip-flop family $\mathfrak{F}=\fF^+\sqcup\fF^-$ associated to $f^{N_f}$  and the number $N_f$ from the flip-flop configuration of the saddle $q_f$  and the blender $(\Gamma_f, \mathfrak{D}_f)$.
The sub-family  $\mathfrak{F}^+$ is formed by  the discs of $\mathfrak{D}_f$. To define
$\mathfrak{F}^-$ let us assume, for simplicity, that $f(q_f)=q_f$.
We consider an auxiliary
 family  $\mathfrak{D}_q$ of $C^1$-embedded discs
 containing $W^\mathrm{u}_{\delta}(q_f,f)$ (for sufficiently small $\delta$) in its interior and consisting of small discs
$D$ such that
\begin{enumerate}
\item[a)]
every $D$ intersects transversely $W^\mathrm{s}_\delta(q_f,f)$ and is tangent to a small cone field around $E^{\mathrm{uu}}$,
\item[b)]
there is $\lambda>1$ such that
 $\|Df(v)\| \ge \lambda \|v\|$
for every vector $v$ tangent to $D$,
\item[c)]
$f(D)$ contains a disc in $\mathfrak{D}_q$.
\end{enumerate}
For the existence of the family $\mathfrak{D}_q$ and its
precise definition
see \cite[Lemma 4.11]{BBD:16}.
It turns out that for every $g$ nearby $f$ the family $\mathfrak{D}_q$ also satisfies these properties for $g$.
We let $\mathfrak{F}^-=\mathfrak{D}_q$.

Observe now that due to the flip-flop configuration $W^u(q_f,f)$ contains a disc $D'\in \mathfrak{D}_f=\mathfrak{F}^+$.
Hence there is large $k_0$ such that for every disc $D\in \mathfrak{D}_f=\mathfrak{F}^+$ and every $N\ge k_0$  the disc $f^N(D)$ contains a sub-disc close enough
to $D'$ and hence contains a disc in $\mathfrak{D}_f=\mathfrak{F}^+$ (note that this family is necessarily open). Note also that $f^N(D)$ contains a disc of $\mathfrak{D}_q$ by
(c). Observe that the choice of $k_0$ holds for every $g$ nearby $f$.
A similar construction holds for the images of the discs in $\mathfrak{D}=\mathfrak{F}^-$, now we use that $W^s(q_f,f)$ transversely intersects every disc
in $\mathfrak{D}_f=\mathfrak{F}^+$. In this way, we get a uniform $N$ in such a way the family
satisfies condition \eqref{i.flipflop2}  in the definition of a flip-flop family (Definition~\ref{i.flipflop33}). In our partially hyperbolic case, condition
\eqref{i.flipflop33}  follows because all the discs we consider are tangent to a strong unstable cone field.

It remains to get condition \eqref{i.flipflop1} on the averages of $(\mathrm{J}_f^{\mrc})_N$. For this, some additional shrinking of the discs of the blender is needed. We will follow \cite[Section 4.4]{BBD:16}.
Note that the map $\mathrm{J}_f^{\mrc}$ is positive for the points in the set $\Gamma_f$ (here we recall that $E^\mathrm{c}$ is expanding
in a neighbourhood  $V$ of $\Gamma_f$ since we consider the metric given by Proposition \ref{p.dandovoltas}).
Consider for each disc $D$ of the family $\mathfrak{D}_f$ a sub-disc $D'$ contained in $V$ such that the family $\mathfrak{D}'_f$ formed by the sets $D'$ is invariant for $f^m$ for some $m$. Again, the same $m$ works for every $g$ sufficiently close to $f$. The precise definition of this new family  $\mathfrak{D}'_f$ is in \cite[Definition 4.15]{BBD:16} and the invariance properties are in \cite[Lemmas 4.17 and 4.18]{BBD:16}.

 Finally, observe that once we have obtained the flip-flop family $\mathfrak{F}=\fF^+\sqcup\fF^-$ the proof that this family has sojourns is exactly as in \cite[Proposition 5.2]{BDB:}.

 This completes our sketch of the proof of Theorem~\ref{t.l.h(q)}.
\end{proof}

\subsection{Proof of Theorem~\ref{t.openanddense}}
The theorem  follows immediately from Theorem~\ref{t.flipfloptailqual} and Theorem~\ref{t.l.h(q)}.

\subsection{Proof of Theorem  \ref{t.average}}
Recall that for a periodic point $p_f$ of $f$ we denote by $\mu_{\cO(p_f)}$ the
unique $f$-invariant probability measure supported on the orbit of $p_f$.
Consider now periodic points $p_f$ and $q_f$ of $f$
satisfying $\int \varphi\, d\mu_{\cO(p_f)}  <0<
\int \varphi \, d\mu_{\cO(q_f)}$.

To prove Theorem~\ref{t.average} it is enough to consider the case where the saddles $p_f$ and $q_f$ have the same index and
are homoclinically related (which is an open and dense condition in $\mathcal{RT}(M)$).
In \cite[Section 5.3]{BDB:} it is explained how the case where the saddles have different indices is reduced to this ``homoclinically related''  case: after an arbitrarily small perturbation of $f$ one gets $g$ with a
saddle-node $r_g$ with $\int \varphi  d\mu_{\cO(r_g)}\ne 0$. Assume that $\int \varphi  d\mu_{\cO(r_g)}>0$. In this case
we perturb $g$ to get $h$ such that $r_h$ has the same index as $p_h$ and these points are homoclinically related. Then we are in the
``homoclinically related case''.

We now prove Theorem~\ref{t.average} when the saddles $p_f$ and $q_f$ are  homoclinically related. Let us  assume for simplicity, that these saddles are fixed points of $f$.
The result follows from the construction in \cite{BDB:},
we will sketch below its main steps.

Recall
that the set $\cD^i(M)$   of
$i$-dimensional (closed) discs $C^1$-embedded in $M$ has a natural topology which is
induced by a metric $\mathfrak{d}$, for details see \cite[Proposition 3.1]{BBD:16}. For  small $\varrho>0$ consider  the
$\varrho$-neighbourhoods
 $\cV^{\mathfrak{d}}_\varrho(p_f)\eqdef \cV^{\mathfrak{d}}_\varrho(W^\mru_{loc}(p_f,f))$ and $\cV^{\mathfrak{d}}_\varrho(q_f)\eqdef
 \cV^{\fd}_\varrho(W^\mru_{loc}(q_f,f))$
of the local unstable manifolds of $p_f$ and $q_f$ for the distance $\fd$ in $\cD^i(M)$, where $i$ is the dimension of the unstable bundle
of $p_f$ and $q_f$.
We consider the following family $\fF_f=\fF_f^+\sqcup \fF_f^-$ of discs:
\begin{itemize}
 \item $\fF_f^-$ is the family of discs in $\cV^\fd_\varrho(p_f)$ contained in $W^\mru(p_f,f)\cup W^\mru(q_f,f)$;
 \item $\fF_f^+$ is the family of discs in $\cV^\fd_\varrho(q_f)$ contained in $W^\mru(p_f,f)\cup W^\mru(q_f,f)$.
\end{itemize}
Note that as $q_f$ and $p_f$ are homoclinically related
these two families are both infinite. Note also that for $\varrho>0$ small enough one has that
$\varphi$ is negative in the discs of $\fF_f^-$ and positive in the discs of $\fF_f^+$.
Note that we can define the families $\fF_g^\pm$ analogously for every $g$ close to  $f$, having also that
$\varphi$ is negative in  $\fF_f^-$ and positive in $\fF_f^+$.

We have the following result which is an improvement of  \cite[Proposition 5.2]{BDB:}. The original result is stated for a single diffeomorphism $f$. Here we have a version valid for a neighbourhood with a uniform control of $n$ in the whole neighbourhood. As in the case in the previous section, this allows us to locally bound the entropy of the measures associated to the flip-flop from below.

\begin{propo}
Consider $f$ and $\varphi$ as above. Then there is $n$ such that
the family $\fF_g$  is a flip-flop family associated to $\varphi$  and $g^n$ and
 has $g$-sojourns  along the homoclinic class $H(p_g,g)$.
 \end{propo}

\begin{proof}
Let us recall the proof of the proposition for $f$ (\cite[Proposition 5.2]{BDB:}).
Since the saddles $p_f$ and $q_f$ are homoclinically related,
there is $n$ such that for every disc  $D\in \fF_f^\pm$ the disc $f^n(D)$ contains
discs  $D_p\in \cV^\fd_\varrho(p_f)$  and $D_q\in \cV^\fd_\varrho(q_f)$. By construction
$D\in W^\mru(p_f,f)\cup W^\mru(q_f,f)$. Observe that for $g$ close enough to $f$ and
 every disc  $D\in \fF_g^\pm$ the disc $g^n(D)$ also contains
discs $D_p\in \cV^\fd_\varrho(p_g)$  and $D_q\in \cV^\fd_\varrho(q_g)$. 

The fact that $\fF_f$ is a flip-flop family is quite
straightforward. The same proof applies to $\fF_g$. For details see \cite[Section 5.2]{BDB:}, where it is also proved that the family has sojourns
in the whole class.
\end{proof}

\bibliographystyle{plain}

\begin{thebibliography}{10}

\bibitem{AbSm}
R.~Abraham and S.~Smale.
\newblock Nongenericity of {$\Omega $}-stability.
\newblock In {\em Global {A}nalysis ({P}roc. {S}ympos. {P}ure {M}ath., {V}ol.
  {XIV}, {B}erkeley, {C}alif., 1968)}, pages 5--8. Amer. Math. Soc.,
  Providence, R.I., 1970.

\bibitem{BBS}
V.~Baladi, Ch. Bonatti, and B.~Schmitt.
\newblock Abnormal escape rates from nonuniformly hyperbolic sets.
\newblock {\em Ergodic Theory Dynam. Systems}, 19(5):1111--1125, 1999.


\bibitem{BR1}
V.~Bergelson and I.~Ruzsa.
\newblock Squarefree numbers, {IP} sets and ergodic theory.
\newblock In {\em Paul {E}rd\H os and his mathematics, {I} ({B}udapest, 1999)},
  volume~11 of {\em Bolyai Soc. Math. Stud.}, pages 147--160. J\'anos Bolyai
  Math. Soc., Budapest, 2002.

\bibitem{BBD:16}
J.~Bochi, Ch. Bonatti, and L.~J. D\'{\i}az.
\newblock Robust criterion for the existence of nonhyperbolic ergodic measures.
\newblock {\em Comm. Math. Phys.}, 344(3):751--795, 2016.

\bibitem{BBD:14}
J. Bochi, Ch. Bonatti, and L.~J. D{\'\i}az.
\newblock Robust vanishing of all {L}yapunov exponents for iterated function
  systems.
\newblock {\em Math. Z.}, 276(1-2):469--503, 2014.

\bibitem{BD-robtran}
Ch. Bonatti and L.~J. D\'{i}az.
\newblock Persistent nonhyperbolic transitive diffeomorphisms.
\newblock {\em Ann. of Math. (2)}, 143(2):357--396, 1996.

\bibitem{BDG}
Ch. Bonatti, L.~J. D\'{i}az, and A.~Gorodetski.
\newblock Non-hyperbolic ergodic measures with large support.
\newblock {\em Nonlinearity}, 23(3):687--705, 2010.

\bibitem{BDPR}
Ch. Bonatti, L.~J. D\'{i}az, E.~R. Pujals, and J.~Rocha.
\newblock Robustly transitive sets and heterodimensional cycles.
\newblock {\em Ast\'erisque}, (286):xix, 187--222, 2003.
\newblock Geometric methods in dynamics. I.

\bibitem{BZ}
Ch. Bonatti and J.~Zhang.
\newblock Periodic measures and partially hyperbolic homoclinic classes.
\newblock To appear in: Ergodic Theory Dynam. Systems.

\bibitem{BDB:}
Ch. Bonatti, L.~J. D{\'\i}az, and Jairo Bochi.
\newblock A criterion for zero averages and full support of ergodic measures.
\newblock {\em Mosc. Math. J.}, 18(1):15--61, 2018.
%

\bibitem{CCGWY}
C.~Cheng, S.~Crovisier, S.~Gan, X.~Wang, and D.~Yang.
\newblock Hyperbolicity versus non-hyperbolic ergodic measures inside
  homoclinic classes.
\newblock To appear in: Ergodic Theory Dynam. Systems.

\bibitem{D-ICM}
L.~J. D{\'\i}az.
\newblock Nonhyperbolic ergodic measures.
\newblock In {\em Proc. Int. Cong. of Math. – 2018 Rio de Janeiro, Vol. 2},
  page 1883–1904.

\bibitem{DG}
L.~J. D\'{i}az and A.~Gorodetski.
\newblock Non-hyperbolic ergodic measures for non-hyperbolic homoclinic
  classes.
\newblock {\em Ergodic Theory Dynam. Systems}, 29(5):1479--1513, 2009.

\bibitem{Downarowicz}
T.~Downarowicz.
\newblock {\em Entropy in dynamical systems}, volume~18 of {\em New
  Mathematical Monographs}.
\newblock Cambridge University Press, Cambridge, 2011.

\bibitem{Downarowicz-Toeplitz}
T. Downarowicz.
\newblock Survey of odometers and {T}oeplitz flows.
\newblock In {\em Algebraic and topological dynamics}, volume 385 of {\em
  Contemp. Math.}, pages 7--37. Amer. Math. Soc., Providence, RI, 2005.

\bibitem{Garling}
D.~J.~H. Garling.
\newblock {\em Analysis on {P}olish spaces and an introduction to optimal
  transportation}, volume~89 of {\em London Mathematical Society Student
  Texts}.
\newblock Cambridge University Press, Cambridge, 2018.

\bibitem{Glasner}
E.~Glasner.
\newblock {\em Ergodic theory via joinings}, volume 101 of {\em Mathematical
  Surveys and Monographs}.
\newblock American Mathematical Society, Providence, RI, 2003.

\bibitem{GIKN}
A.~S. Gorodetski\u\i, Yu.~S. Il'yashenko, V.~A. Kleptsyn, and M.~B. Nalski\u\i.
\newblock Nonremovability of zero {L}yapunov exponents.
\newblock {\em Funktsional. Anal. i Prilozhen.}, 39(1):27--38, 95, 2005.

\bibitem{KN}
V.~A. Kleptsyn and M.~B. Nalski\u\i.
\newblock Stability of the existence of nonhyperbolic measures for
  {$C^1$}-diffeomorphisms.
\newblock {\em Funktsional. Anal. i Prilozhen.}, 41(4):30--45, 96, 2007.

\bibitem{Krengel}
U.~Krengel.
\newblock On certain analogous difficulties in the investigation of flows in a
  probability space and of transformations in an infinite measure space.
\newblock In {\em Functional {A}nalysis ({P}roc. {S}ympos., {M}onterey,
  {C}alif., 1969)}, pages 75--91. Academic Press, New York, 1969.

\bibitem{KL}
D.~Kwietniak and M.~{\L}{\c{a}}cka.
\newblock Feldman-{K}atok pseudometric and the {GIKN} construction of
  nonhyperbolic ergodic measures.
\newblock preprint, 2017.

\bibitem{KLO}
D. Kwietniak, M. {\L}{\c{a}}cka, and Piotr Oprocha.
\newblock Generic points for dynamical systems with average shadowing.
\newblock {\em Monatsh. Math.}, 183(4):625--648, 2017.

\bibitem{Martha}
M.~{\L}{\c{a}}cka.
\newblock {\em Zastosowanie pseudometryk dynamicznie generowanych oraz
  kr\'olewskie miary niehiperboliczne}.
\newblock PhD thesis, Jagiellonian University in Krak\'ow, 2018.

\bibitem{LM}
D. Lind and B. Marcus.
\newblock {\em An introduction to symbolic dynamics and coding}.
\newblock Cambridge University Press, Cambridge, 1995.

\bibitem{Mda}
R.~Ma\~n\'e.
\newblock Contributions to the stability conjecture.
\newblock {\em Topology}, 17(4):383--396, 1978.

\bibitem{Nadkarni}
M.~G. Nadkarni.
\newblock {\em Basic ergodic theory}.
\newblock Birkh\"auser Advanced Texts: Basler Lehrb\"ucher. [Birkh\"auser
  Advanced Texts: Basel Textbooks]. Birkh\"auser Verlag, Basel, second edition,
  1998.

\bibitem{New}
S.~E. Newhouse.
\newblock The abundance of wild hyperbolic sets and nonsmooth stable sets for
  diffeomorphisms.
\newblock {\em Inst. Hautes \'Etudes Sci. Publ. Math.}, (50):101--151, 1979.

\bibitem{Petersen}
K.~Petersen.
\newblock {\em Ergodic theory}, volume~2 of {\em Cambridge Studies in Advanced
  Mathematics}.
\newblock Cambridge University Press, Cambridge, 1989.
\newblock Corrected reprint of the 1983 original.

\bibitem{PuSa}
E.~R. Pujals and M.~Sambarino.
\newblock Homoclinic tangencies and hyperbolicity for surface diffeomorphisms.
\newblock {\em Ann. of Math. (2)}, 151(3):961--1023, 2000.

\bibitem{Shields}
P.~C. Shields.
\newblock {\em The ergodic theory of discrete sample paths}, volume~13 of {\em
  Graduate Studies in Mathematics}.
\newblock American Mathematical Society, Providence, RI, 1996.

\bibitem{Sh}
M.~Shub.
\newblock Topologically transitive diffeomorphisms on $\mathbb{T}^4$.
\newblock In {\em Proceedings of the {S}ymposium on {D}ifferential {E}quations
  and {D}ynamical {S}ystems}, volume 206 of {\em Lecture Notes in Math.},
  page~39. Springer, Berlin, 1971.

\bibitem{Wal:82}
P.~Walters.
\newblock {\em An introduction to ergodic theory}, volume~79 of {\em Graduate
  Texts in Mathematics}.
\newblock Springer-Verlag, New York-Berlin, 1982.

\bibitem{WZ}
X~Wan and Zhang J.
\newblock Ergodic measures with multi-zero lyapunov exponents inside homoclinic
  classes.
\newblock {\em Preprint {\tt arXiv:1604.03342}.}

\bibitem{Weiss}
B.~Weiss.
\newblock {\em Single orbit dynamics}, volume~95 of {\em CBMS Regional
  Conference Series in Mathematics}.
\newblock American Mathematical Society, Providence, RI, 2000.

\end{thebibliography}

\end{document}